\documentclass[a4paper,english,11pt]{smfart}
\usepackage[dvips]{graphicx}
\usepackage{amssymb}
\usepackage{amsmath}
\usepackage[utf8]{inputenc}
\usepackage[T1]{fontenc}
\usepackage[francais]{babel}
\usepackage{graphicx}
\usepackage{amsmath}
\usepackage{amsthm}
\usepackage{amsfonts}
\usepackage{amssymb}
\usepackage{enumerate}
\usepackage{vaucanson-g}
\usepackage{graphicx}
\usepackage{setspace}
\usepackage{enumitem}
\usepackage{dsfont}

\newtheorem{thm}{Theorem}[section]
\newtheorem{prop}[thm]{Proposition}
\newtheorem{lem}[thm]{Lemma}
\newtheorem{coro}[thm]{Corollary} 

\newtheorem{conj}[thm]{Conjecture}
\theoremstyle{definition}
\newtheorem{defi}[thm]{Definition}

\theoremstyle{remark} 
\newtheorem{rem}[thm]{Remark}
\theoremstyle{definition} 
\newtheorem{ex}[thm]{Example}

\newtheorem{prob}[thm]{Problem}
\newtheorem{fact}[thm]{Fact}

\newcounter{constant}
\newcommand{\newconstant}[1]{\refstepcounter{constant}\label{#1}}
\newcommand{\useconstant}[1]{c_{\ref{#1}}}
\newcommand{\defconstant}[1]{ \newconstant{c_{#1}}\expandafter\newcommand\csname c#1\endcsname{\useconstant{c_{#1}}} }  
\newcounter{Constant}
\newcommand{\newConstant}[1]{\refstepcounter{Constant}\label{#1}}
\newcommand{\useConstant}[1]{C_{\ref{#1}}}
\newcommand{\defConstant}[1]{ \newConstant{C_{#1}}\expandafter\newcommand\csname C#1\endcsname{\useConstant{C_{#1}}} }

\newcommand{\Q}{{\overline{\mathbb Q}}}
\newcommand{\V}{\mathcal V}

\newcommand{\y}{{\boldsymbol{y}}}

\newcommand{\Z}{\mathbb{Z}}

\newcommand{\N}{\mathbb{N}}

\newcommand{\M}{\mathcal M}

\renewcommand{\L}{\widetilde{\mathbb L}}

\newcommand{\x}{{\boldsymbol{x}}}

\newcommand{\bK}{{\rm \bf{K}}}

\newcommand{\z}{{\boldsymbol{z}}}
\newcommand{\lambd}{{\boldsymbol{\lambda}}}
\newcommand{\f}{{\boldsymbol f}}
\renewcommand{\O}{\mathcal{O}}

\newcommand{\balpha}{{\boldsymbol{\alpha}}}
\newcommand{\X}{\boldsymbol{X}}

\newcommand{\bmu}{{\boldsymbol{\mu}}}

\newcommand{\bbeta}{{\boldsymbol{\beta}}}

\numberwithin{equation}{section} 

\title[Mahler's method in several variables II]{Mahler's method in several variables II: Applications to base change problems and finite automata}

\author{Boris Adamczewski}
\address{
Univ Lyon, Universit\'e Claude Bernard Lyon 1\\
 CNRS UMR 5208, Institut Camille Jordan \\
 F-69622 Villeurbanne Cedex, France}
\email{Boris.Adamczewski@math.cnrs.fr}

\author{Colin Faverjon}
\email{colin.faverjon@riseup.net}
\date{}

\thanks{This project has received funding from the European Research Council (ERC) under the European Union's Horizon 2020 
research and innovation programme under the Grant Agreement No 648132. }
\begin{document}

\begin{abstract}  
This is the second part of a work devoted to the study of linear Mahler systems in several variables from the 
perspective of transcendence and algebraic independence. 
From the lifting theorem obtained in the first part, we first derive a general result, showing that Mahler functions 
in several variables, associated with transformations having multiplicatively dependent spectral radii,  
take algebraic independent values at algebraic points  
provided that these points are sufficiently independent. 
Then, we focus on applications of this result and of the two main results of Part I of this work.  
Our main application concerns problems about the representation of natural and real numbers in 
integer bases involving automata theory. 
These can be translated in terms of algebraic relations over $\Q$ between values of Mahler functions in 
one variable. 
We also apply our results to the algebraic independence of Mahler functions and their specializations, 
and to the study of the values of Hecke-Mahler series. 
\end{abstract}

\bibliographystyle{abbvr}
\maketitle
\setcounter{tocdepth}{1}
\tableofcontents 

\section{Introduction}\label{Introduction}

This is the second part of a work devoted to the study of linear Mahler systems in several variables from the 
perspective of transcendence and algebraic independence.  
In the first part \cite{AF3}, we prove two main results concerning regular singular systems: 
the \emph{lifting theorem} \cite[Theorem 2.1]{AF3} and the \emph{purity theorem} \cite[Theorem 2.4]{AF3}.  
Let $(f_1(\z),\ldots,f_m(\z))\in \Q\{\z\}^m$ be a vector representing a solution to  
a regular singular Mahler system, and let $\balpha$ be some suitable algebraic point. The lifting theorem 
says that any homogeneous algebraic relation over $\Q$ between the complex numbers $f_1(\balpha),\ldots,f_m(\balpha)$  
can be lifted to a similar algebraic relation over $\Q(\z)$ between the functions $f_1(\z),\ldots,f_m(\z)$.  The study of the algebraic 
(resp.\ linear) relations between the values of such Mahler functions can thus be reduced to the easier study of the algebraic (resp.\ linear) 
relations between these functions. Results of this nature are a principal goal of transcendence theory.  However, we stress that \emph{easier} does not 
necessarily mean \emph{easy}, and, so far, only the linear relations between Mahler functions in one variable have been fully understood \cite{AF1,AF2}. 
The purity theorem is of different nature. It states that values of Mahler functions associated with sufficiently independent matrix 
transformations behave independently.  
Let $r\geq 2$ be an integer and, for every $i$, $1\leq i\leq r$, let $(f_{i,1}(\z),\ldots,f_{i,m_i}(\z))\in \Q\{\z\}^{m_i}$ be a vector representing a solution to  
a regular singular Mahler system associated with a matrix transformation $T_i$. Furthermore, let us assume that the spectral radii of the transformations $T_i$ 
are pairwise multiplicatively independent. Then the purity theorem 
says that the study of the algebraic relations between the values of all these functions at suitable (possibly different) algebraic points 
can be reduced to the study of each system separately. Furthermore, the latter can be done using the lifting theorem.  
We emphasize that such a miracle turns out to be a consequence of the formalism introduced by Mahler, 
which makes possible to deal with systems in several variables.  
We recall now two well-known advantages that this formalism also offers. 
 
\begin{itemize}

\item[{\rm (A)}] The first advantage of this formalism is that it allows us to deal with the algebraic relations over $\Q$ between the values of a Mahler function at different algebraic points.    
We stress that this is a natural goal of such a theory. In the setting of Siegel $E$-functions, the study of an $E$-function at different points can 
be achieved by considering different $E$-functions at the same point. 
Indeed, if $f(z)$ is an $E$-function and $\alpha$ is an algebraic number, the function $f(\alpha z)$ is still an $E$-function.  
This trick  no longer works with Mahler functions. Fortunately,    
the theory in several variables allows us to overcome this deficiency.      
Let us give a simple example. With the function $\mathfrak f(z)=\sum_{n=0}^{\infty} z^{2^n}$, we can associate  
the two variables linear system 
\begin{equation*}
\begin{pmatrix}
1 \\ \mathfrak f(z_1^{2}) \\ \mathfrak f(z_2^{2})
\end{pmatrix} =   
\begin{pmatrix} 
1 & 0 & 0\\
-z_1 & 1& 0 \\
-z_2 & 0 & 1 \\
 \end{pmatrix} \begin{pmatrix}
 1 \\ \mathfrak  f(z_1) \\  \mathfrak f(z_2)
 \end{pmatrix} \, .
\end{equation*}
The underlying transformation matrix 
\begin{equation*}
T=\begin{pmatrix}
2& 0\\
0 & 2 \\
 \end{pmatrix}\end{equation*} 
belongs to the class $\mathcal M$. Furthermore, the point $\balpha=(1/2,1/3)$ is regular with respect to this system, 
and the pair $(T,\balpha)$ is admissible. 
Now, the key point is that the transcendence of $\mathfrak f(z)$ gives for free the algebraic independence over $\overline{\mathbb Q}(z_1,z_2)$ 
of the functions $\mathfrak  f(z_1)$ and $\mathfrak f(z_2)$. By the lifting theorem, we obtain that 
$\mathfrak f(1/2)$ and $\mathfrak f(1/3)$ are algebraically independent over $\Q$.  To sum-up: 
 \emph{transcendence results in Mahler's method automatically lead to algebraic independence results}.

\medskip

\item[{\rm (B)}] The second advantage of this formalism is that it allows us to deal with the values of a larger 
class of one-variable analytic functions in $\Q\{z\}$ obtained via 
some suitable specializations of Mahler functions in several variables. Mahler's favorite example is the so-called Hecke-Mahler function  
$$
\mathfrak f_{\omega}(z)=\sum_{n=0}^{\infty}\lfloor n\omega\rfloor z^n \,,
$$ 
where $\omega$ is a quadratic irrational real number. Though the function $\mathfrak f_{\omega}(z)$ is not expected to be a Mahler function,  
we have that $\mathfrak f_{\omega}(z)= F_{\omega}(z,1)$, where 
$$
F_{\omega}(z_1,z_2)=\sum_{n_1=0}^{\infty} \sum_{n_2=0}^{\lfloor n_1\omega\rfloor}z_1^{n_1}z_2^{n_2}
$$ 
turns out to be a Mahler function in two variables.  
 In a different direction, Cobham \cite{Co68} proved that the generating functions of morphic sequences 
can always be obtained as specializations of the form $F(z,z,\ldots,z)$ of multivariate Mahler functions $F(z_1,\ldots,z_n)$. 

\end{itemize}

In the direction of (A), we first show that the lifting theorem implies another purity theorem, 
namely Theorem \ref{thm: purete2}.  In contrast to Theorem 2.4 of \cite{AF3}, it applies to 
Mahler functions associated with matrix transformations having the same spectral radius. 
The independence of the matrix transformations required in Theorem 2.4 is replaced by asking 
for some sort of independence for the different points at which the functions coming from each 
Mahler system are evaluated. 

\medskip

\begin{thm}[Purity--Independent points]\label{thm: purete2} 
We continue with the notation of Theorem 2.4 of the first part \cite{AF3}.  
Let $r\geq 2$ be an integer and $\rho>1$ be a real number. 
For every integer $i$, $1\leq i \leq r$, we consider a regular singular Mahler system   
\begin{equation}\stepcounter{equation}
\label{eq:Mahleri}\tag{\theequation .i}
\left(\begin{array}{c} f_{i,1}(T_i\z_i) \\ \vdots \\ f_{i,m_i}(T_i\z_i) \end{array} \right) = 
A_i(\z_i)\left(\begin{array}{c} f_{i,1}(\z_i) \\ \vdots \\ f_{i,m_i}(\z_i) 
\end{array} \right) 
\end{equation}
where $A_i(\z_i)$ belongs to ${\rm GL}_{m_i}(\Q(\z_i))$, $\z_i:=(z_{i,1},\ldots,z_{i,n_i})$ is a family of indeterminates, 
$T_i$ is an $n_i\times n_i$ matrix  
with non-negative  integer coefficients and with spectral radius 
$\rho$. For every $i$, $1\leq i\leq r$,  let us consider a set 
$$\mathcal E_i\subseteq \{f_{i,1}(\balpha_i),\ldots,f_{i,m_i}(\balpha_i)\}$$  
and set $\mathcal E:=\cup_{i=1}^r\mathcal E_i$. 
Suppose that 

\begin{enumerate}

\item[\rm{(i)}]  for every $i$, $\balpha_i\in(\Q^\star)^{n_i}$ is a regular point 
with respect to the system \eqref{eq:Mahleri} and the pair $(T_i,\balpha_i)$ is admissible, and 

\item[\rm{(ii)}]   the point $\balpha=(\balpha_1,\ldots,\balpha_r)$ is $T$-independent, where 
$$
T = \left(\begin{array}{ccc} T_1 & & \\ & \ddots & \\ && T_r \end{array}\right) \, . 
$$

\end{enumerate}
Then
$$
{\rm Alg}_{\Q}(\mathcal E) = \sum_{i=1}^r {\rm Alg}_{\Q}(\mathcal E_i\mid \mathcal E)\,.
$$
\end{thm}

\begin{rem}
If the matrices $T_1,\ldots,T_r$ have different, but pairwise multiplicatively dependent, spectral radii $\rho_i$, 
we could pick positive integers $d_1,\ldots,d_r$ such that 
$$
\rho_1^{d_1}=\cdots=\rho_r^{d_r}\, , 
$$
and iterate each system \eqref{eq:Mahleri} $d_i$ times, in order to apply Theorem \ref{thm: purete2} 
with the matrices $T_i^{d_i}$.
\end{rem}

In the sequel, we refer to Theorem 2.4 of \cite{AF3} as the \emph{first purity theorem} and to 
Theorem \ref{thm: purete2} as the \emph{second purity theorem}. We thus have at our 
disposal three main theorems from which we derive our different applications. 
We emphasize that although these three results concern Mahler systems in several variables, 
we mainly focus here on applications concerning analytic functions of a single variable, according to (A) and (B).  
As an illustration of possible applications in the direction of (B), we obtain for instance the following result about the values 
of Hecke--Mahler functions, extending theorems of Ku. Nishioka \cite{Ni94} 
and Masser \cite{Mas99}. 

\begin{thm}\label{th:HeckeMahler}
Let $\omega_{1},\ldots,\omega_{r}$ be distinct quadratic irrational real numbers 
 such that the quadratic fields $\mathbb Q(\omega_1),\ldots,\mathbb Q(\omega_r)$ are all distinct. 
For every $i$, $1\leq i\leq r$, let $\alpha_{i,1},\ldots,\alpha_{i,m_i}$ 
be distinct algebraic numbers with $0<\vert \alpha_{i,j}\vert < 1$. 
Then the numbers 
$$
f_{\omega_{i}}(\alpha_{i,j}),\, 1\leq i \leq r,\, 1 \leq j \leq m_i\, ,
$$ 
are algebraically independent over $\Q$.
\end{thm}

\begin{rem}\label{rem:contre-exemple_Hecke-Mahler}
This result is almost the best possible, in the sense that if $\omega_1$ and $\omega_2$ belong to 
the same quadratic number field, then there may be some relations between the values of 
$f_{\omega_1}(z)$ and $f_{\omega_2}(z)$. For instance, if $\omega$ is 
a positive real number, then $f_{\omega}(z)+f_{\omega}(-z)-2f_{2\omega}(z^2) = 0$. 
Hence, $f_{\omega}(1/2)$, $f_{\omega}(-1/2)$, and $f_{2\omega}(1/4)$ are linearly 
dependent over $\Q$. 
\end{rem}

Beyond Hecke--Mahler series, our main application concerns Mahler functions in one variable. 
We recall that, given an integer $q\geq 2$, $f(z)\in \overline{\mathbb Q}\{z\}$ is a {\it $q$-Mahler function} 
if there exist polynomials $p_0(z),\ldots , p_n(z)\in \overline{\mathbb Q}[z]$, not all zero,  
such that  
\begin{equation} \label{eq:Mahler1}
 p_0(z)f(z)+p_1(z)f(z^q)+\cdots + p_n(z)f(z^{q^n}) \ = \ 0. 
 \end{equation}  
If $f(z)$ is $q$-Mahler for some $q$, we simply say that $f(z)$ is a Mahler function. 
In Section \ref{sec: basechange}, we describe several problems, namely Problems  \ref{conj: strongf}, \ref{conj: weakv}, and \ref{conj: strongv},  
concerning expansions of natural and real numbers in integer bases.  
These problems all involve the so-called \emph{automatic sequences}, and 
all are also widely open.  They take their roots in the works of Cobham \cite{Co68,Co69} in the late sixties, and of Loxton and van der Poorten 
\cite{LvdP770,LvdP78,LvdP82,vdP76,vdP87} in the late seventies and in the eighties. 
As recalled in Section \ref{sec: basechange}, 
the generating function associated with a $q$-automatic sequence is a $q$-Mahler function, so that,  
in the end, a solution to all these problems would follow from the following general conjecture.  
We recall that given complex numbers $\alpha_1,\ldots,\alpha_r$ are said to be multiplicative independent if there is no non-zero 
tuple of integers $n_1,\ldots,n_r$ such that 
$\alpha_1^{n_1}\cdots\alpha_r^{n_r}=1$. 

\begin{conj}\label{conj: strongvm}
Let $r\geq 2$ be an integer. For every integer $i$, $1\leq  i\leq r$, we let $q_i\geq 2$ be an integer, 
$f_i(z)\in \Q\{z\}$ be a $q_i$-Mahler function, 
and $\alpha_i$ be an algebraic number, $0 <\vert\alpha_i\vert <1$, such that $f_i(z)$ is well-defined at $\alpha_i$.  Then the following 
properties hold. 

\begin{itemize}

\item[{\rm (i)}] Let us assume that $\alpha_1,\ldots, \alpha_r$ are multiplicatively independent. Then 
the numbers $f_1(\alpha_1),f_2(\alpha_2),\ldots,f_r(\alpha_r)$ are algebraically independent over $\Q$ if and only if 
they are all transcendental.  

\medskip

\item[{\rm (ii)}] Let us assume that $q_1,\ldots, q_r$ are pairwise multiplicatively independent. Then 
the numbers $f_1(\alpha_1),f_2(\alpha_2),\ldots,f_r(\alpha_r)$ are algebraically independent over $\Q$ if and only if 
they are all transcendental.  
\end{itemize}
\end{conj}

Our main contribution towards Conjecture  \ref{conj: strongvm} is the following.

\begin{defi}\label{defi:regular_singular_une_variable}
A $q$-Mahler function is  \emph{regular singular}  if it is 
the coordinate of a vector representing a solution to a regular singular Mahler system. 
\end{defi}

\begin{thm}\label{thm:main}
Conjecture \ref{conj: strongvm} is true if each $f_i(z)$ is regular singular.  
\end{thm}

\begin{rem}
In fact, a $q$-Mahler function $f(z)$ is regular singular if and only if the Mahler system associated 
with the companion matrix of Equation \eqref{eq:Mahler1} is regular singular (see Remark \ref{rem:ers}).  
We stress that if $p_0(0)p_n(0)\not=0$ in Equation \eqref{eq:Mahler1}, 
the corresponding $q$-Mahler function $f(z)$ is regular singular. 
Thus, being regular singular is a generic property for $q$-Mahler functions. 
In a previous paper \cite{AF1}, the authors provide an algorithm to determine whether or not the numbers $f_i(\alpha_i)$ 
occurring in Conjecture \ref{conj: strongvm} are transcendental. 
Furthermore, Richardson \cite[Theorem 2]{Ri01} provides an algorithm to determine  whether or not the algebraic numbers $\alpha_1,\ldots, \alpha_r$ 
are multiplicatively independent.  
The algebraic independence of the numbers $f_1(\alpha_1),\ldots,f_r(\alpha_r)$ in Theorem \ref{thm:main} 
can thus be determined effectively.  
\end{rem}

We mention the following consequence of Theorem \ref{thm:main} related to Problems 
\ref{conj: strongf} and \ref{conj: strongf2}. 

\begin{coro}
\label{coro:function}
Let $q_1,q_2,\ldots,q_r$ be pairwise multiplicatively independent positive integers. For every $i$, $1\leq i\leq r$, let
$f_i(z)\in \Q\{z\}$ be a regular singular $q_i$-Mahler function that is not a rational function.    
Then $f_1(z),f_2(z),\ldots,f_r(z)$ are algebraically independent over $\Q(z)$. 
\end{coro}

Unfortunately, generating functions of automatic sequences are not always regular singular, and, 
consequently, Problems \ref{conj: strongf}, \ref{conj: weakv}, and 
\ref{conj: strongv} remain open. However, Theorem \ref{thm:main} and Corollary \ref{coro:function} 
mark significant progress towards their resolution. 

\medskip

This paper is organized as follows.  
In Section \ref{sec: basechange}, we state Problems \ref{conj: strongf}, \ref{conj: weakv}, and \ref{conj: strongv}, which 
were at the origin of our interest in Mahler's method. 
Section \ref{sec:purete} is devoted to the proof Theorem \ref{thm: purete2}. Theorem \ref{thm:main} and 
Corollary \ref{coro:function} are proved in Section \ref{sec:1variable}. In Section \ref{sec:several}, we discuss Mahler functions in 
several variables, generalizing Corollary \ref{coro:function} to this wider framework. In Section \ref{sec:spe}, we consider one-variable analytic functions 
obtained as specializations of Mahler functions in several variables. We define the notion of a 
\emph{good specialization} and prove yet another extension of 
Corollary \ref{coro:function} to this setting. Applications of our results to Hecke--Mahler series are given in Section \ref{sec: HeckeMahler}. 
We prove there Theorem \ref{th:HeckeMahler}, as well as two complementary results. 
In Section \ref{sec:ex} we show, through a final example, how our three main results can be combined together to derive algebraic independence of 
values of classical Mahler functions. Similar examples can be produced at will. 

\section{Two base change problems involving finite automata}\label{sec: basechange}

In this section, we first briefly recall some informal definitions of an automatic sequence and of an automatic 
set of natural numbers. We refer  the reader to the book of Allouche and Shallit \cite{AS} for more details.   
Then, we describe several base change problems involving these two notions.  

\subsection{Automatic sequences and automatic sets} 

Let $q \ge 2$ be an integer.
An infinite sequence ${\bf a}=(a_n)_{n\geq 0}$ is 
said to be $q$-automatic if $a_n$ is a finite-state function of the base-$q$ 
representation of $n$. This means that there exists a deterministic finite automaton with output 
(DFAO) taking the
 base-$q$ expansion of $n$ as input and producing the term $a_n$ as 
output. We say that a sequence is generated by a finite automaton, or for short is automatic,  
if it is $q$-automatic for some $q$. 

\begin{ex}\label{ex:ThueMorse}
One of the most famous example of a $2$-automatic sequence is the Thue--Morse sequence 
\begin{equation*}
{\boldsymbol{\mathfrak{tm}}}= 01 10 10 01 10 01 01\cdots \, ,
\end{equation*}
which is defined as follows.  Its $n$th term is equal to $0$ if the sum of the binary digits of $n$ is even, 
and it is equal to $1$ otherwise (see Figure \ref{AB:figure:thue}). 

\begin{figure}[htbp]
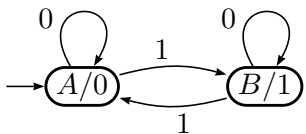

\centering
\VCDraw{%
        \begin{VCPicture}{(0,-0.5)(4,1.5)}
 \StateVar[A/0]{(0,0)}{a}  \StateVar[B/1]{(4,0)}{b}
\Initial[w]{a}
\LoopN{a}{0}
\LoopN{b}{0}
\ArcL{a}{b}{1}
\ArcL{b}{a}{1}
\end{VCPicture}
}
\caption{A $2$-automaton generating the Thue--Morse sequence.}
  \label{AB:figure:thue}
\end{figure}
\end{ex}

A set ${\mathcal E}\subset \mathbb N$ is said to be \emph{$q$-automatic} if its characteristic sequence, defined by 
$a_n=1$ if $n\in {\mathcal E}$ and by $a_n=0$ otherwise, is a $q$-automatic sequence. 
This means that there exists a 
DFAO taking the
 base-$q$ expansion of $n$ as input and accepting 
this natural number (producing as output the symbol $1$) if  $n$ belongs to 
${\mathcal E}$. Otherwise, this automaton rejects $n$, producing as output the symbol $0$. 

\begin{ex}
The set $\{1,2,4,8,16,\ldots\}$ formed by the powers of $2$ is a typical example 
of a $2$-automatic set (see Figure \ref{AB:figure:2n}). 
\begin{figure}[h]
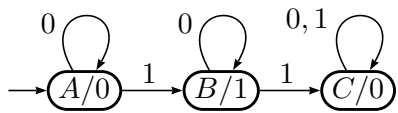

\centering 
\VCDraw{  \begin{VCPicture}{(0,-0.5)(6,2)}
    \StateVar[A/0]{(0,0)}{A} \StateVar[B/1]{(3,0)}{B}   \StateVar[C/0]{(6,0)}{C}  
    \Initial{A} 
     \EdgeL{A}{B}{1}  \EdgeL{B}{C}{1} \LoopN{A}{0} 
    \LoopN{B}{0}  \LoopN{C}{0,1}
  \end{VCPicture}}
  \caption{A $2$-automaton recognizing the powers of $2$.}
  \label{AB:figure:2n}
\end{figure}

\end{ex}

\subsection{Expansions of natural numbers in integer bases}

The proposition according to which a natural number is divisible by $9$ if and only if  
the sum of its digits (in decimal expansion) is itself divisible by $9$ is one of the most notorious arithmetic properties. 
Though sometimes more intricate, there are similar rules about the divisibility by $2,3,5,11$...   
Already in the seventieth century, the mathematician and philosopher Pascal \cite{Pas} addressed 
this problem in a general setting: 
\emph{I shall also set out a
general method which allows one to
discover, by simple inspection of its digits,
whether a number is divisible by an
arbitrary other number; this method
applies not only to our decimal system of
numeration (which system rests on a
convention, an unhappy one besides, and
not on a natural necessity, as the vulgar
think), but it also applies without fails to
every system of numeration having for base
whatever number one wishes, as may be
discovered in the following pages.} 
In a modern terminology, the existence of such simple divisibility rules in every integer base 
can be reformulated as follows.

\begin{fact}\label{fact:1} Let $a,b$ be non-negative integers. Then the arithmetic progression $a\mathbb N+b$ 
is a $q$-automatic set for all integers $q \geq 2$.  
\end{fact}

 \begin{figure}[h]
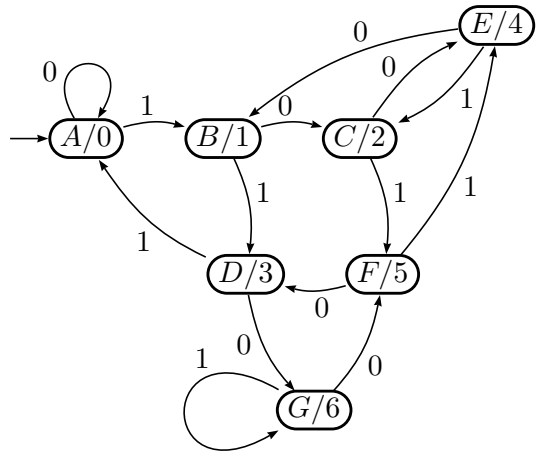

\centering
\VCDraw{%
        \begin{VCPicture}{(0,3)(9,-3.5)}
 \StateVar[A/0]{(0,0)}{0}  \StateVar[B/1]{(3,0)}{1}  \StateVar[C/2]{(6,0)}{2} \StateVar[E/4]{(9,2.5)}{4} 
 \StateVar[D/3]{(3.5,-3)}{3} \StateVar[F/5]{(6.5,-3)}{5} \StateVar[G/6]{(5,-6)}{6}
\Initial[w]{0}
\LoopN{0}{0}
\LoopW{6}{1}
\ArcL{0}{1}{1}
\ArcL{1}{2}{0}
\ArcL{2}{4}{0}
\ArcL{3}{0}{1}
\ArcL{5}{3}{0}
\ArcL{4}{2}{1}
\ArcL{1}{3}{1}
\ArcL{2}{5}{1}
\ArcR{5}{4}{1}
\ArcR{6}{5}{0}
\ArcR{3}{6}{0}
\ArcR{4}{1}{0}
\end{VCPicture}
}\caption{A $2$-automaton computing $n \bmod 7$.}
\end{figure}

Beyond divisibility rules, and this is not a great surprise, there usually does not exist any automatic test to determine the 
main arithmetical properties of natural numbers. 
For instance, prime numbers,  
perfect squares, and square-free numbers are not $k$-automatic sets, and this whatever the base $k$ chosen to represent 
the natural numbers. 
A notable exception is given by the set of natural numbers that can be written as the sum of three squares. 
Indeed, it follows from a theorem of Legendre that this set is $2$-automatic (see \cite{Co72}).

The first base change problem we consider is the following one. 
Though it is obvious to determine whether a binary natural number is a power of $2$,   
it seems more difficult to identify this property from its decimal expansion. 
This intuition can be formalized by showing that the set $\{2^n \mid n\geq 0\}$ 
is $2$-automatic, while it is not $10$-automatic.   
In 1969, Cobham \cite{Co69} proved the following fundamental theorem, 
solving completely the problem of the base-dependence for all automatic sets. 

\begin{thm}[Cobham] Let $q_1$ and $q_2$ be two multiplicatively independent natural numbers.  
A set $\mathcal E\subset \mathbb N$ is both $q_1$- and $q_2$-automatic if and only if it is a finite 
union of arithmetic progressions. 
\end{thm}

\begin{rem}
In addition, we recall that if a set is $q_1$-automatic, then it is also $q_2$-automatic for all integers $q_2$  multiplicatively dependent with $q_1$.  
\end{rem}

In other words, divisibility rules are the only \emph{automatic rules} 
whose existence does not depend on the 
base. In more algebraic terms,  we expect that Cobham's theorem 
can be strengthened as follows.

\begin{prob}\label{conj: strongf}
Let $q_1$ and $q_2$ be two multiplicatively independent natural numbers.  
Let $\mathcal E_1$ be a $q_1$-automatic set and $\mathcal E_2$ be a $q_2$-automatic set.  
Prove that if the generating functions
$$
f_1(z)=\sum_{n\in\mathcal E_1} z^n \quad\mbox{and} \quad f_2(z)=\sum_{n\in\mathcal E_2} z^n 
$$
are both not rational, then they are algebraically independent over $\Q(z)$.  
\end{prob}

\subsection{Computational complexity of real numbers, finite automata, and base dependence}

Similar questions occur when replacing sequences of natural numbers by real numbers.  
However, these are often much harder to handle.   
We consider the computational complexity of real numbers with respect to a given integer base $b$. 
The most simple class is formed by the \emph{automatic real numbers}, that is, those whose base-$b$ expansion can be 
generated by a finite automaton.  
The analogue of Fact \ref{fact:1} reads as follows.    

\begin{fact}
For all integers $b\geq 2$, the base-$b$ expansion of a rational number can be generated by a finite automaton.  
\end{fact}

In this setting, the study of classical sequences of natural numbers, such as prime numbers, perfect squares, 
and square-free numbers is replaced by 
the study of classical irrational mathematical constants such as $\sqrt 2$ and $\pi$.  
This is an old source of frustration for mathematicians. While these numbers have very simple 
geometric descriptions, their decimal 
expansions 
$$
\begin{array}{cccc}
&\langle \sqrt 2\rangle_{10} &= &1.414\,213\, 562\, 373\, 095\, 048\, 801\, 688\, 724\, 209\, 698\, 078\, 569 \cdots \\ 
\text{ and} & &&\\

&  \langle \pi\rangle_{10}& = &3.141\, 592\, 653\, 589\, 793\, 238\, 462\, 643\, 383\, 279\, 502\,884\, 197  \cdots 
\end{array}
$$
remain totally mysterious. In this area, a major problem is to prove that none of the numbers $\sqrt 2,\pi,e,\log 2$ 
is automatic.  This problem is still widely open, but Bugeaud and the first author \cite{ABu07} proved that 
no algebraic irrational real number is automatic\footnote{Cassaigne and the first author \cite{AC06} 
also proved that Liouville numbers are not automatic. This was extended to Mahler's $U$-numbers 
by Bugeaud and the first auhtor \cite{ABu11}.}.  
On the other hand, automatic irrational real numbers, such as  
the binary Thue--Morse number 
$$
\langle \tau \rangle_2 = 0.011\, 010\, 011\, 001\, 011\, 010\, 010\, 110\, 011\, 010\, 011\, 001\, 011 \cdots\,,
$$
do exist. Though $\tau$ has a simple binary expansion, its decimal expansion  

$$
\langle \tau \rangle_{10} = 0.412\, 454 \, 033 \,640\, 107\, 597\, 783\, 361\, 368\, 258\, 455\, 283\, 089\cdots 
$$
seems much more unpredictable. This leads us to consider our second base change problem.     
Problem \ref{conj: weakv} below is a stronger form of Problem 7 in \cite[p. 403]{AS}. 
It can be thought of as the analogue of Cobham's theorem in this setting. 

\begin{prob}\label{conj: weakv}
Let $b_1$ and $b_2$ be two multiplicatively independent natural numbers. 
Prove that a real number is automatic in both bases $b_1$ and $b_2$ if and only if it is a 
rational number. 
\end{prob}

\begin{rem}
Using some classical results about automatic sequences, it can be shown that if 
a real number is automatic in base $b$, then it is also automatic in all bases that are multiplicatively dependent with $b$.  
\end{rem}

We also consider the following much stronger version of Problem \ref{conj: weakv}. 

\begin{prob}\label{conj: strongv}
Let $r\geq 1$ be an integer. Let $b_1,\ldots,b_r$ be pairwise multiplicatively independent natural numbers, 
and, for every $i$, $1\leq i \leq r$, 
let  $\xi_i$ be an irrational real number whose base-$b_i$ expansion can be generated by a finite automaton.  
Prove that the numbers $\xi_1,\ldots,\xi_r$ are algebraically independent over $\Q$. 
\end{prob}

Problem \ref{conj: weakv} is only solved for $r=1$ in \cite{ABu07}. 

\subsection{Connection with Mahler's method}

The following fundamental connection between finite automata and Mahler functions was noticed by Cobham in 1968 \cite{Co68}. 

\begin{itemize}

\item[{\rm (C)}] Let $q\geq 2$ be an integer and ${\bf a}=(a_n)_{n\geq 0}$ be a $q$-automatic sequence with values in $\Q$.  
Then the generating function  
$$
f_{\bf a}(z):= \sum_{n=0}^{\infty}a_nz^n 
$$
is a $q$-Mahler function. 
\end{itemize}

Our main problems can thus be extended as problems concerning Mahler functions.  
For instance, Property (C) led Loxton and van der Poorten (see \cite{vdP87}) to conjecture the following generalization of Cobham's theorem, 
which was later proved by Bell 
and the first author \cite{AB}.    

\begin{thm}[A. and Bell]\label{thm: ABe} 
Let $q_1$ and $q_2$ be two multiplicatively independent positive integers. A power series $f(z)\in \Q[[z]]$ is both 
$q_1$- and $q_2$-Mahler if and only if it is a rational function. 
\end{thm}

Recently, Sch\"afke and Singer \cite{SS1,SS2} give a totally different proof of Theorem \ref{thm: ABe} based on the Galois theory of difference equations 
associated with the Mahler operators $\sigma_q: z \mapsto z^q$. The great advantage of this new proof is that it does not make use 
of Cobham's theorem. 
Our last problem, which generalizes Problem \ref{conj: strongf} and Theorem \ref{thm: ABe}, reads as follows.

\begin{prob}\label{conj: strongf2}
Let $r\geq 1$ be an integer. Let $q_1,\ldots,q_r$ be pairwise multiplicatively independent natural numbers, 
and, for every $i$, $1\leq i \leq r$, let  $f_i(z)\in\Q[[z]]$ be a $q_i$-Mahler function that is not a rational function.  
Prove that $f_1(z),\ldots,f_r(z)$ are algebraically independent over $\Q(z)$. 
\end{prob}

Problem \ref{conj: strongf2} is only solved for $r=1$ (see \cite[Theorem 5.1.7]{Ni_Liv}). 
Recently, a partial solution to the case $r=2$ was obtained in \cite{ADH} 
using the Galois theory of parametrized difference equations.  
Using (C), it is easy to check that part (i) of Conjecture \ref{conj: strongvm} would allow us to 
solve Problems \ref{conj: weakv} and \ref{conj: strongv}. 
Furthermore, the proof of Corollary \ref{coro:function} given in Section \ref{sec:1variable} shows 
that part (ii) of Conjecture \ref{conj: strongvm} would 
allow us to solve 
Problems \ref{conj: strongf} and \ref{conj: strongf2}.

\section{Proof of Theorem \ref{thm: purete2}}\label{sec:purete}

In this section, we show how to deduce our second purity theorem from the lifting theorem. 
We first prove Theorem \ref{thm: purete2} when all the sets $\mathcal E_i$ have maximal cardinality.  

\begin{lem}\label{lem:puritysystementiers}
We continue with the assumptions of Theorem \ref{thm: purete2}. If for every $i$, 
$1\leq i \leq r$, one has 
$$
\mathcal E_i= \{f_{i,1}(\balpha_i),\ldots,f_{i,m_i}(\balpha_i)\}\, ,
$$
then  
$$
{\rm Alg}_{\Q}(\mathcal E) = \sum_{i=1}^r {\rm Alg}_{\Q}(\mathcal E_i \mid \mathcal E)\, .
$$
\end{lem}

\begin{proof}
Set $\z:=(\z_1,\ldots,\z_r)$, and let us consider the block diagonal $T$-Mahler system 
$$
\left(\begin{array}{c}f_{1,1}(\z_1)\\ \vdots \\ f_{1,m_1}(\z_1) \\  \vdots \\ \vdots \\ f_{r,1}(\z_r) \\ \vdots \\ f_{r,m_r}(\z_r) \end{array} \right)
=
\left(\begin{array}{cccccccc} & & & & & & & \\ & A(\z_1)& &&& && \\ &&&&&& &\\ &&& \ddots &&&& \\&&&& \ddots &&& \\ &&&&&&& 
\\ &&&&&& A(\z_r) & \\ &&&&&&& \end{array} \right)
\left(\begin{array}{c}f_{1,1}(T\z_1)\\ \vdots \\ f_{1,m_1}(T\z_1) \\ \vdots \\ \vdots \\ f_{r,1}(T\z_r) \\ \vdots \\ f_{r,m_r}(T\z_r) \end{array} \right)\, .
$$
We also consider $r$ families of indeterminates 
$$
\X_1:=(X_{1,1},\ldots,X_{1,m_1}),\ldots,\X_r:=(X_{r,1},\ldots,X_{r,m_r})\, ,
$$
and set $\X:=(\X_1,\ldots,\X_r)$. 
We are going to prove that ${\rm Alg}_{\Q}(\mathcal E) \subset \sum_{i=1}^r {\rm Alg}_{\Q}(\mathcal E_i \mid \mathcal E)$, 
the converse inclusion being trivial. 
Let $P(\X) \in {\rm Alg}_{\Q}(\mathcal E)$.  
We infer from the lifting theorem that there exists a polynomial $Q \in \Q[\z,\X]$, 
such that 
$$
Q(\balpha_1,\ldots,\balpha_r,\X)=P(\X) \qquad \text{ and } \qquad Q(\z,f_1(\z_1),\ldots,f_m(\z_r))=0 \, .
$$
Now, since the families of variables $\z_1,\z_2,\ldots,\z_r$ are all disjoint, Lemma 8.2 of \cite{AF3} implies that there exists a decomposition 
$Q=Q_1+\cdots+Q_r$ in $\Q[\z,\X]$ such that 
$$
Q_i(\z,\X_1,\ldots,\X_{i-1},f_{i,1}(\z_i),\ldots,f_{i,m_i}(\z_i),\X_{i+1},\ldots,\X_r)=0 \,,
$$
for every $i$, $1\leq i \leq r$. 
Setting $P_i=Q_i(\balpha_1,\ldots,\balpha_r,\X)$, we obtain that $P=P_1+\cdots+P_r$. Furthermore, 
one has 
$$
P_i(\X_1,\ldots,\X_{i-1},f_{i,1}(\balpha_i),\ldots,f_{i,m_i}(\balpha_i),\X_{i+1},\ldots,\X_r)=0 \, ,
$$
for every $i$, $1\leq i\leq r$. Hence, $P_i \in {\rm Alg}_{\Q}(\mathcal E_i \mid \mathcal E)$, which ends the proof.
\end{proof}

Before proving Theorem \ref{thm: purete2}, we need the following simple lemma about transcendence degrees.

\begin{lem}\label{lem:transcendancedegree}
Let $\mathcal E_1, \ldots,\mathcal E_r,\mathcal F_1,\ldots, \mathcal F_r$ be non-empty finite sets of complex numbers, 
such that $\mathcal E_i \subset \mathcal F_i$, for every $i$, $1 \leq i \leq r$. Let us assume that  
$$
{\rm tr.deg}_{\Q}\left(\bigcup_{i=1}^r \mathcal F_i\right) = \sum_{i=1}^r {\rm tr.deg}_{\Q}(\mathcal F_i)\, .
$$
Then 
$$
{\rm tr.deg}_{\Q}\left(\bigcup_{i=1}^r \mathcal E_i\right) = \sum_{i=1}^r {\rm tr.deg}_{\Q}(\mathcal E_i)\, .
$$
\end{lem}

\begin{proof}
We prove this lemma by using a descending induction on the size of the sets $\mathcal E_i$. 
When $\mathcal E_i = \mathcal F_i$ for all $i$, there is nothing to prove. 
Let us now assume that there exists an index $i_0$ such that $\mathcal E_{i_0} \subsetneq \mathcal F_{i_0}$, and such that the theorem is proved for larger  
$\mathcal E_{i_0}$, the other sets $\mathcal E_i$ being unchanged. 
Without loss of generality, we assume that $i_0=1$. We pick a number $\xi \in \mathcal F_1 \setminus \mathcal E_1$, and set 
$\mathcal E_1'=\mathcal E_1 \cup \{\xi\}$. We consider two different cases. First, we assume that $\xi$ 
is algebraic over $\Q(\mathcal E_1)$. Then  
$$
{\rm tr.deg}_{\Q}(\mathcal E_1')={\rm tr.deg}_{\Q}(\mathcal E_1)\, ,
$$
and, as $\xi$ is also algebraic over $\Q\left(\bigcup_{i=1}^r \mathcal E_i\right)$, we deduce that 
$$
{\rm tr.deg}_{\Q}\left(\mathcal E_1' \cup \bigcup_{i=2}^r \mathcal E_i\right)={\rm tr.deg}_{\Q}\left(\bigcup_{i=1}^r \mathcal E_i\right)\, .
$$
By assumption, we thus obtain that 
\begin{eqnarray*}
{\rm tr.deg}_{\Q}\left(\bigcup_{i=1}^r \mathcal E_i\right) & = & {\rm tr.deg}_{\Q}\left(\mathcal E_i' \cup \bigcup_{i=2}^r \mathcal E_i\right)
\\ & = & {\rm tr.deg}_{\Q}(\mathcal E_i') + \sum_{i=2}^r {\rm tr.deg}_{\Q}(\mathcal E_i)
\\ & = & \sum_{i=1}^r {\rm tr.deg}_{\Q}(\mathcal E_i)\, ,
\end{eqnarray*}
as wanted. 
Now, we assume that $\xi$ is transcendental over $\Q(\mathcal E_i)$. Then
$$
{\rm tr.deg}_{\Q}(\mathcal E'_1)={\rm tr.deg}_{\Q}(\mathcal E_1) + 1\, .
$$
By assumption, we deduce that $\xi$ is also transcendental over $\Q\left(\bigcup_{i=1}^r \mathcal E_i\right)$. 
Then
\begin{eqnarray*}
{\rm tr.deg}_{\Q}\left(\bigcup_{i=1}^r \mathcal E_i\right)& = & {\rm tr.deg}_{\Q}\left(\mathcal E_i' \cup \bigcup_{i=2}^r \mathcal E_i\right)-1
\\ & = & {\rm tr.deg}_{\Q}(\mathcal E_i') -1 + \sum_{i=2}^r {\rm tr.deg}_{\Q}(\mathcal E_i)
\\ & = &   \sum_{i=1}^r {\rm tr.deg}_{\Q}(\mathcal E_i)\, ,
\end{eqnarray*}
as wanted. This ends the proof. 
 \end{proof}
 
Theorem \ref{thm: purete2} is now a direct consequence of lemmas \ref{lem:puritysystementiers} and \ref{lem:transcendancedegree}.

\begin{proof}[Proof of Theorem \ref{thm: purete2}]
Note that the inclusion 
\begin{equation}\label{eq:inclusionindirect}
\sum_{i=1}^r {\rm Alg}_{\Q}(\mathcal E_i \mid \mathcal E) \subset  {\rm Alg}_{\Q}(\mathcal E)\, ,
\end{equation}
is trivial. It is thus enough to prove that $\sum_{i=1}^r {\rm Alg}_{\Q}(\mathcal E_i \mid \mathcal E)$ is a prime ideal whose height 
is larger than or equal to the one of ${\rm Alg}_{\Q}(\mathcal E)$. Given a prime ideal $\mathfrak p$ of a ring, we let ${\rm ht}(\mathfrak p)$ 
denote the height $\mathfrak p$, that is, the maximal length of a chain of prime ideal included in $\mathfrak p$. 

Set 
$$
\mathcal F_i := \{f_{i,1}(\balpha_i),\ldots,f_{i,m_i}(\balpha_i)\},\, 1 \leq i \leq r\, ,
$$
and $\mathcal F = \cup_i \mathcal F_i$. By Lemma \ref{lem:puritysystementiers}, we know that 
$$
{\rm Alg}_{\Q}(\mathcal F) = \sum_{i=1}^r {\rm Alg}_{\Q}(\mathcal F_i \mid \mathcal F)\, .
$$

We stress that
\begin{equation}
\label{eq:equalityheight}
{\rm ht}\left({\rm Alg}_{\Q}(\mathcal F)\right) =\sum_{i=1}^r {\rm ht}\left({\rm Alg}_{\Q}(\mathcal F_i)\right)\, ,
\end{equation}
where ${\rm Alg}_{\Q}(\mathcal F) \subset \Q[\X]$ and ${\rm Alg}_{\Q}(\mathcal F_i) \subset \Q[\X_i]$. Indeed, 
from Krull's height theorem, the height of the prime ideal ${\rm Alg}_{\Q}(\mathcal E) \subset \Q[\X]$ is equal to the size of a minimal set of generators of 
${\rm Alg}_{\Q}(\mathcal E)$ in the Noetherian ring $\Q[\X]$. 
For every $i$, $1 \leq i \leq r$, let $P_{i,1},\ldots,P_{i,h_i} \in \Q[\X_i]$ denote a minimal system of generators of ${\rm Alg}_{\Q}(\mathcal F_i)$ in $\Q[X_i]$. 
Hence ${\rm ht}\left({\rm Alg}_{\Q}(\mathcal F_i)\right) = h_i$. From lemma \ref{lem:puritysystementiers}, the family
$$
P_{1,1},\ldots,P_{1,h_1},P_{2,1},\ldots,P_{r,h_r}
$$
spans ${\rm Alg}_{\Q}(\mathcal F)$ in $\Q[\X]$, which gives that 
$$
{\rm ht}\left({\rm Alg}_{\Q}(\mathcal F)\right) \leq \sum_{i=1}^r {\rm ht}\left({\rm Alg}_{\Q}(\mathcal F_i)\right)\, .
$$
The converse inequality is trivial. 
Furthermore, the height of the ideal ${\rm Alg}_{\Q}(\mathcal F)$ satisfies
\begin{equation}
\label{eq:height_transdeg}
{\rm ht}\left({\rm Alg}_{\Q}(\mathcal F)\right) = m-{\rm tr.deg}_{\Q}(\mathcal F)\, .
\end{equation}
Equalities \eqref{eq:equalityheight} and \eqref{eq:height_transdeg} thus imply that 
\begin{equation}
\label{eq:transdeg}
{\rm trans.deg}(\mathcal F) = \sum_{i=1}^r {\rm trans.deg}(\mathcal F_i)\, .
\end{equation}
Going back to the sets $\mathcal E_i$, Lemma \ref{lem:transcendancedegree} now implies that 
\begin{equation}
\label{eq:egalitetranscdeg}
{\rm trans.deg}_{\Q}\left(\mathcal E \right) = \sum_{i=1}^r {\rm trans.deg}_{\Q}(\mathcal E_i)\, .
\end{equation}
It thus follows from \eqref{eq:height_transdeg} that 
$$
{\rm ht}\left({\rm Alg}_{\Q}(\mathcal E)\right) =\sum_{i=1}^r {\rm ht}\left({\rm Alg}_{\Q}(\mathcal E_i)\right)\, .
$$
Set $\mathcal I := \sum_{i=1}^r {\rm Alg}_{\Q}(\mathcal E_i \mid \mathcal E)$. Then the isomorphism 
$$
\frac{\Q[\X_1]}{{\rm Alg}_{\Q}(\mathcal E_1)} \otimes_{\Q} \cdots \otimes_{\Q} \frac{\Q[\X_r]}{{\rm Alg}_{\Q}(\mathcal E_r)} \simeq \frac{Q[\X]}{\mathcal I}\,  
$$
implies that $\mathcal I$ is a prime ideal. Indeed, the tensor product of integral domains, over an algebraically closed field, is an integral domain. 
Furthermore, ${\rm ht}(\mathcal I)=\sum_{i=1}^r{\rm ht}({\rm Alg}_{\Q}(\mathcal E_i \mid \mathcal E))$, since the dimension of the product of affine varieties 
is equal to the sum of the dimension of these varieties. 
It follows that ${\rm Alg}_{\Q}(\mathcal E)$ and $\sum_{i=1}^r {\rm Alg}_{\Q}(\mathcal E_i \mid \mathcal E)$ are both prime ideals with the same height. 
By \eqref{eq:inclusionindirect}, these two ideals are equal. This ends the proof.
\end{proof}

\section{Mahler functions in one variable}\label{sec:1variable} 

In this section, we consider Mahler functions in one variable. Our main aim is to prove Theorem \ref{thm:main} and 
Corollary \ref{coro:function}, but we start with a short discussion about few general principles governing 
 the study of these functions. The following three fundamental principles serve as a mantra for number theorists working in Mahler's method.  

\begin{itemize}

\item[{\rm (I)}] Transcendental $q$-Mahler functions  take transcendental values  at algebraic points.

\medskip

\item[{\rm (II)}] Algebraically independent $q$-Mahler functions take algebraically independent values at algebraic points. 

\medskip

\item[{\rm (III)}] Linearly independent $q$-Mahler functions take linearly independent values at algebraic points.

\end{itemize}

Of course, these must be taken with a pinch of salt. For instance, if $f(z)$ is a transcendental $q$-Mahler function, 
so is $g(z)=(z-1/2)f(z)$, and $g(1/2)=0$ is not a transcendental number.  Even more subtle counter-examples such as 
$$
f(z)=\prod_{n=0}^{\infty}(1-2z^{3^n})
$$ 
can be cooked up easily.  
However, these three principles can be rigorously established in the following sense.  
Let $r<1$ be a positive real number, then there exists a finite set $\mathcal E$ (depending on $r$ and the 
corresponding $q$-Mahler functions) such that Principles (I)--(III) are satisfied for all algebraic numbers $\alpha$, 
$0<\vert \alpha\vert\leq r$, that does not belong to $\mathcal E$. For Principles (I) this is a consequences of 
Nishioka's theorem, as observed by Becker \cite[Lemma 6]{Bec94}. In fact, his argument extends to show that Principles (II) is also a 
consequence of Nishioka's theorem (see Proposition \ref{prop:becker}).    
For Principle (III), the more recent works of Philippon \cite{PPH} 
and the authors \cite{AF1,AF2} are needed.  
Furthermore, the authors \cite{AF1,AF2}  show that the exceptional set $\mathcal E$ in 
Principles (I) and (III) can be effectively determined.  
In contrast, the following two additional principles, which do not fall under the scope of Mahler's method in one variable, 
have not yet been established. 

\begin{itemize}

\item[{\rm (IV)}]  Transcendental Mahler functions take algebraically independent values at multiplicatively independent algebraic points. 

\medskip

\item[{\rm (V)}] Transcendental Mahler functions associated with pairwise multiplicatively independent transformation take algebraically independent values 
at algebraic points. In particular, they are algebraically independent over $\Q(z)$. 
    
\end{itemize}

\noindent
Again, we can make them rigorous as follows. Let $r<1$ be a positive real number, then there exists a finite set $\mathcal E$ (depending on $r$ and the 
corresponding Mahler functions, say $f_1(z),\ldots,f_n(z)$) such that Principles IV and V are satisfied for all $n$-tuples 
of algebraic numbers $\alpha_1,\ldots,\alpha_n$, with $0<\vert \alpha_1\vert,\ldots, \vert\alpha_n\vert\leq r$  
and $\alpha_1,\ldots,\alpha_n\not\in\mathcal E$. We stress that Theorem \ref{thm:main} and Corollary \ref{coro:function} validate 
Principles (IV) and (V) in the case of regular singular Mahler functions.  

\begin{prop}\label{prop:becker} 
Let $f_1(z),\ldots,f_m(z)\in\Q\{z\}$ be analytic functions that converge on a connected open set $\mathcal U\subset \mathbb C$. 
Let $\mathcal A\subset \mathcal U$ be a set such that 
$$
{\rm tr.deg}_{\Q}(f_1(\alpha),\ldots,f_m(\alpha))= {\rm tr.deg}_{\Q(z)}(f_1(z),\ldots,f_m(z)) \, ,
$$
for all $\alpha\in \mathcal A$. 
If the functions $f_1(z),\ldots,f_\ell(z)$ are algebraically independent over $\Q(z)$, then the set 
$$
\mathcal E=\{\alpha\in \mathcal A \mid f_1(\alpha),\ldots,f_\ell(\alpha) \mbox{ are algebraically dependent over }\Q\}
$$
has finite intersection with any compact subset of $\mathcal U$.  
\end{prop}

\begin{proof} We follow the argument of Becker \cite[Lemma 6]{Bec94}. 
Let us assume that ${\rm tr.deg}_{\Q(z)}(f_1(z),\ldots,f_m(z))=r$. By assumption, $r\geq \ell$. 
Reordering if necessary, we can assume without loss of generality 
that $f_1(z),\ldots,f_r(z)$ are algebraically independent over $\Q(z)$. 
Let us consider an integer $j_0>r$. Then $f_{j_0}(z)$ is algebraic over the field 
$\Q(z)(f_1(z),\ldots,f_r(z))$ and there exists a non-trivial relation of the form 
$$
\sum_{i=0}^{d_{j_0}}A_{i,{j_0}}(z,f_1(z),\ldots,f_r(z)) f_{j_0}(z)^i=0 \, ,
$$
where the polynomials $A_{i,j_0}\in \Q[z,X_1,\ldots,X_r]$ are not all zero.  
Let $i_0$ be such that $A_{i_0,j_0}(z,X_1,\ldots,X_r)\not=0$.  
Since the function $f_1(z),\ldots,f_r(z)$ are algebraically independent over $\Q(z)$, we get that 
$A_{i_0,j_0}(z,f_1(z),\ldots,f_r(z))$ is a non-zero function that is analytic on $\mathcal U$. 
Now, let $\mathcal C$ denote a compact subset of $\mathcal U$. 
Then there exists a finite set $\mathcal E_{j_0}$ such that  
$A_{i_0,j_0}(\alpha,f_1(\alpha),\ldots,f_r(\alpha))\not=0$  for all $\alpha\in  \mathcal C\setminus \mathcal E_{j_0}$. For such $\alpha$, 
the number $f_{j_0}(\alpha)$ is algebraic over the field $\Q(f_1(\alpha,\ldots,f_r(\alpha))$. 
 Then for all 
$\alpha\in\mathcal C\setminus \cup_{j=r+1}^m\mathcal E_j$, we have that
$$
{\rm tr.deg}_{\Q}(f_1(\alpha),\ldots,f_m(\alpha))={\rm tr.deg}_{\Q}(f_1(\alpha),\ldots,f_r(\alpha))\,.
$$
By definition of $\mathcal A$, we deduce that ${\rm tr.deg}_{\Q}(f_1(\alpha),\ldots,f_r(\alpha))=r$ for all $\alpha\in\mathcal A\cap \mathcal C\setminus (\cup_{j=r+1}^m\mathcal E_j)$. 
In particular, $f_1(\alpha),\ldots,f_\ell(\alpha)$ 
are algebraically independent over 
$\Q$. It follows that $\mathcal E\cap \mathcal C\subset \cup_{j=r+1}^m\mathcal E_j$ is a finite set. This ends the proof. 
\end{proof}

\begin{proof}[Proof of Theorem \ref{thm:main}] 

Let $f(z)$ be a regular singular $q$-Mahler function and $\alpha$ be a non-zero algebraic number 
such that $f(z)$ is well-defined at $\alpha$. We are going to show that there exists a $q$-Mahler function $g(z)$ 
such that the following properties hold. 

\begin{itemize}
\item[{\rm (a)}]  $g(\alpha)=f(\alpha)$.

\medskip

\item[{\rm (b)}] The function $g(z)$ is the coordinate of a solution vector of regular singular Mahler system 
\begin{equation*}
\left(\begin{array}{c} g_1(z)=g(z) \\ \vdots \\ g_{m}(z) \end{array}\right)
=B(z)\left(\begin{array}{c} g_{1}(z^q) \\ \vdots \\ g_{m}(z^q) \end{array}\right)\, .
\end{equation*}

\medskip

\item[{\rm (c)}]  The point $\alpha$ is regular with respect to this system. 

\end{itemize}

By assumption, the function $f(z)$ is the coordinate of a regular singular Mahler system, say 
\begin{equation}
\label{eq:Mahlerunevariable}
\left(\begin{array}{c} f_{1}(z) \\ \vdots \\ f_{m}(z) \end{array}\right)
=A(z)\left(\begin{array}{c} f_{1}(z^q) \\ \vdots \\ f_{m}(z^q) \end{array}\right)\, .
\end{equation}
Up to a reordering of the index, we assume without loss of generality that $f_1(z)=f(z)$, $f_1(z)\ldots,f_r(z)$ are linearly independent over $\Q(z)$, 
and that the functions $f_{r+1}(z),\ldots,f_m(z)$ belong to ${\rm Vect}_{\Q(z)} \{f_1(z),\ldots,f_r(z)\}$. 
Applying a rational gauge transform to \eqref{eq:Mahlerunevariable}, 
we obtain a new Mahler system
\begin{equation}
\label{eq:Mahleruneconjugué}
\left(\begin{array}{c}  f_{1}(z) \\ \vdots \\ f_{r}(z) \\ 0 \\ \vdots \\0 \end{array}\right)
=\left(\begin{array}{ccc|ccc} &&&&&
\\ & A_1(z) & & & A_2(z) &
\\ &&&&&
\\ \hline & &&&&
\\ & A_3(z) &&& A_4(z) &
\\ &&&&& \end{array}\right)
\left(\begin{array}{c}  f_{1}(z^q) \\ \vdots \\ f_{r}(z^q) \\ 0 \\ \vdots \\0  \end{array}\right)\, .
\end{equation}
Since the functions $f_1(z),\ldots,f_r(z)$ are linearly independent over $\Q(z)$, so are $f_1(z^q),\ldots,f_r(z^q)$. 
Hence, $A_3(z)$ is a zero matrix. The system \eqref{eq:Mahleruneconjugué} remains regular singular for  
any rational gauge transform preserves this property.    
Thus, there exist an invertible matrix $\Phi(z)$ with coefficients in $\widehat{\bK}=\cup_{d\geq 1}\Q\{z^{1/d}\}$, 
and a constant matrix $B$  such that
$$
B=\Phi(z)\left(\begin{array}{c|c} 
 A_1(z) & A_2(z) 
\\ \hline 0 & A_4(z)  \end{array}\right) \Phi(z^q)^{-1}\, .
$$
Up to a constant gauge transform, we can assume that $B$ is a lower triangular matrice. Hence, we 
obtain that 
$$
\left(\begin{array}{c|c} 
 B_1 & 0
\\ \hline B_3 & B_4  \end{array}\right)
\left(\begin{array}{c|c} 
   \Phi_1(z^q) & \Phi_2(z^q) 
\\ \hline \Phi_3(z^q) & \Phi_4(z^q)  \end{array}\right)
=
\left(\begin{array}{c|c} 
  \Phi_1(z) & \Phi_2(z) 
\\ \hline \Phi_3(z) & \Phi_4(z)  \end{array}\right)
\left(\begin{array}{c|c} 
 A_1(z) & A_2(z) 
\\ \hline 0 & A_4(z)  \end{array}\right) \,.
$$
Identifying the left upper squares, we get that 
$$
B_1=\Phi_1(z)A_1(z)\Phi_1(z^q)^{-1}\, .
$$
Hence, the system
\begin{equation}
\label{eq:Mahleruneconjuguéreduit}
\left(\begin{array}{c}  f_{1}(z) \\ \vdots \\ f_{r}(z) \end{array}\right)
=A_1(z)
\left(\begin{array}{c} f_{1}(z^q) \\ \vdots \\ f_{r}(z^q) \end{array}\right)\, .
\end{equation}
is regular singular. 
Since the functions $f_1(z),\ldots,f_r(z)$ are linearly independent, we infer from \cite[Theorem 1.10]{AF1} 
that there exists an integer $l$ 
such that the numbers $\alpha_1^{q^l}$ is regular for the system
\begin{equation}
\label{eq:Mahleriteree}
\left(\begin{array}{c} f_{1}(z) \\ \vdots \\ f_{r}(z) \end{array}\right)
=A_1^{(l)}(z)
\left(\begin{array}{c}f_{1}(z^{q^l}) \\ \vdots \\ f_{r}(z^{q^l}) \end{array}\right)\, ,
\end{equation}
where
$$
A_1^{(l)}(z)=A_1(z)A_1(z^q)\cdots A_1(z^{q^{l-1}})\, ,
$$
and such that $\alpha$ is not a pole of $A_1^{(l)}(z)$. 
The definition of $A_1^{(l)}$ ensures that this new system remains regular singular. 
Let $(a_1(z),\ldots,a_{r}(z))$ denote the first row of $A^{(l)}_1(z)$. 
Set  
$$
g(z)= a_{1}(\alpha)f_1(z^{q^l})+\cdots+ a_{r}(\alpha)f_r(z^{q^{l}})\, . 
$$
Then $g(z)$ is a constant linear combination of the functions $f_1(z^{q^l}),\ldots,f_r(z^{q^{l}})$. 
Since the point $\alpha^{q^l}$ is regular with respect to the system \eqref{eq:Mahleriteree}, 
there exists a constant gauge transform of \eqref{eq:Mahleriteree} turning $g(z)$ into the first coordinate 
of a regular singular Mahler system 
with respect to which $\alpha$ is a regular point. Furthermore, we infer from \eqref{eq:Mahleriteree} 
that $g(\alpha)=f(\alpha)$, as wanted. 
This ends the first part of the proof.

Let us prove the case (i) of Theorem \ref{thm:main}.  We assume that, for every $i$, $1\leq i\leq r$, the number $f_i(\alpha_i)$ is transcendental. 
With each pair $(f_i(z),\alpha_i)$, we can associate a $q_i$-Mahler function $g_i(z)$ satisfying Conditions (a), (b), and (c).  
Let us divide the natural numbers $1,\ldots,r$ into $s$ classes $\mathcal I_1,\ldots,\mathcal I_s$ so that 
if  $i$ and $j$ belong to two different classes then $q_i$ and $q_j$ are multiplicatively independent. Iterating the systems associated with the functions $g_i$ 
a suitable number of times, we can assume without loss of generality that $q_i=q_j$ when $i$ and $j$ belong to the same class.  
Set 
\begin{eqnarray*}
\mathcal E &=& \{g_1(\alpha_1),\ldots,g_r(\alpha_r)\} \, \\
&=&  \{f_1(\alpha_1),\ldots,f_r(\alpha_r)\}
\end{eqnarray*}
and $\mathcal E_i=\{g_j(\alpha_j)\mid j\in \mathcal I_i\}$.   
Then the first purity theorem implies that 
\begin{equation}\label{eq:Es}
{\rm Alg}_{\Q}(\mathcal E) = \sum_{i=1}^s {\rm Alg}_{\Q}(\mathcal E_i \mid \mathcal E)\, .
\end{equation}
Now, let us fix $i\in\{1,\ldots,s\}$. We have $\mathcal E_i=\{g_{i_1}(\alpha_{i_1}),\ldots,g_{i_t}(\alpha_{i_t})\}$ for some distinct integers 
$ i_1,\ldots,i_t$, $1\leq i_1,\ldots,i_t\leq r$.  
Set $T = q_i{\rm I}_t$, where we let ${\rm I}_t$ denote the identity matrix of size $t$. 
By assumption, the numbers $\alpha_{i_1},\ldots,\alpha_{i_t}$ are multiplicatively independent. This is equivalent to the $T$-independence of the $t$-tuple 
$(\alpha_{i_1},\ldots,\alpha_{i_t})$. 
For every $j$, $1\leq j \leq t$, we set $\mathcal E_{i,j}=\{g_{i_j}(\alpha_{i_j})\}$. 
Now, we can  apply the second purity theorem. We obtain that  
\begin{equation}\label{eq:Eij}
{\rm Alg}_{\Q}(\mathcal E_i) = \sum_{j=1}^t {\rm Alg}_{\Q}(\mathcal E_{i,j} \mid \mathcal E_i)\, .
\end{equation}
By assumption, ${\rm Alg}_{\Q}(\mathcal E_{i,j}\mid \mathcal E_i)=0$ for all $j$, $1\leq j\leq t$, since 
$f_{i_j}(\alpha_{i_j})=g_{i_j}(\alpha_{i,j})$ is transcendental. Then it follows from \eqref{eq:Eij} that ${\rm Alg}_{\Q}(\mathcal E_i)=0$. 
Finally, \eqref{eq:Es} implies that ${\rm Alg}_{\Q}(\mathcal E)=\{0\}$. 
In other words, the numbers $f_1(\alpha_1),\ldots,f_r(\alpha_r)$ are algebraically independent, as wanted.  

Now, we prove the case (ii) of Theorem \ref{thm:main}. 
We assume that, for every $i$, $1\leq i\leq r$, the number $f_i(\alpha_i)$ is transcendental. 
As previously, we associate with each pair $(f_i(z),\alpha_i)$ a function $g_i(z)$ satisfying Conditions (a), (b), and (c).  
Since the natural numbers $q_i$ are pairwise multiplicatively independent, we can apply the first purity theorem to these systems. 
Indeed, each system 
is associated here with a one-dimensional matrix transformation $T_i=(q_i)$ and the spectral radius of such matrix is just $q_i$.   
Setting 
$$\mathcal E = \{g_1(\alpha_1),\ldots,g_r(\alpha_r)\} \, 
$$ 
and $\mathcal E_i=\{g_i(\alpha_i)\}$, for every $i$, $1\leq i\leq r$, we deduce that 
$$
{\rm Alg}_{\Q}(\mathcal E) = \sum_{i=1}^r {\rm Alg}_{\Q}(\mathcal E_i \mid \mathcal E)\, .
$$
Again, ${\rm Alg}_{\Q}(\mathcal E_i \mid \mathcal E)=0$ for every $i$, $1\leq i\leq r$, since 
by assumption $f_i(\alpha_i)$ is transcendental. 
This shows that ${\rm Alg}_{\Q}(\mathcal E) =\{0\}$ and 
we conclude as previously that the numbers $f_1(\alpha_1),\ldots,f_r(\alpha_r)$ are algebraically independent.  
This ends the proof. 
\end{proof}

Now, we prove Corollary $\ref{coro:function}$. 

\begin{proof}[Proof of corollary \ref{coro:function}]
By Theorem \ref{thm:main}, we just have to prove that there exists an algebraic number $\alpha$, $0<\vert \alpha\vert<1$, 
such that the functions $f_i(z)$ are all well-defined and 
transcendental at $\alpha$.  Indeed, choosing $\alpha_1=\cdots =\alpha_r=\alpha$, 
Theorem \ref{thm:main} implies that the numbers $f_1(\alpha),\ldots,f_r(\alpha)$ 
are algebraically independent. Hence, the functions $f_1(z),\ldots,f_r(z)$ are algebraically independent.   
Let $i\in\{1,\ldots,r\}$. 
Let $\rho<1$ be a positive real number and let $B(0,\rho)$ denote the close complex disc of radius $\rho$. 
Becker \cite[Lemma 6]{Bec94} deduced from Nishioka's theorem that there are only finitely many points $\alpha \in B(0,\rho)$ 
such that $f_i(\alpha)$ is algebraic. Furthermore, the function $f_i(z)$ have only a finite number of poles in $B(0,\rho)$.  
So, for all but finitely many algebraic numbers $\alpha$ in $B(0,r)$,  all the functions $f_i(z)$ are well-defined at $\alpha$ and $f_i(\alpha)$ is transcendental. 
This ends the proof. 
\end{proof}

\section{Mahler functions in several variables} \label{sec:several}

This section is devoted to Mahler functions in several variables. 
We first define the notion of a regular singular $T$-Mahler function, after which 
we extend Theorem \ref{thm:main} and Corollary \ref{coro:function} to this setting. 

\begin{defi}\label{defi:regular_singular} 
Let $\z=(z_1,\ldots,z_n)$ be a vector of indeterminates and let $T$ be $n\times n$ matrix with non-negative integer 
coefficients. A function $f(\z)\in\Q\{\z\}$ is a $T$-Mahler function if there exist polynomials 
$p_0(\z),\ldots , p_n(\z)\in \overline{\mathbb Q}[\z]$, not all zero,  
such that  
\begin{equation*}
 p_0(\z)f(z)+p_1(\z)f(T\z)+\cdots + p_n(\z)f(T^n\z) \ = \ 0 \,
 \end{equation*}  
 or, equivalently, if there exists a linear $T$-Mahler system 
 \begin{equation*}
\left(\begin{array}{c} f_{1}(\z) \\ \vdots \\ f_{m}(\z) \end{array}\right)
= A(\z)\left(\begin{array}{c} f_{1}(T\z) \\ \vdots \\ f_{m}(T\z) \end{array}\right)\, 
\end{equation*}
with $f(\z)=f_{1}(\z)$ and $A(\z) \in {\rm Gl}_{m}(\Q(\z))$. 
A $T$-Mahler function is 
said to be \emph{regular singular} if it is the coordinate of a vector representing a solution 
to a regular singular $T$-Mahler system.
\end{defi}

\begin{rem}\label{rem:ers}
Note that if $f(\z)$ is a regular singular $T$-Mahler function, then every system having a solution 
represented by a vector of analytic functions 
containing $f(\z)$ and whose coordinates are linearly independent over $\Q(\z)$ is also regular singular. 
A sketch of proof of this fact is given in the proof of Theorem \ref{thm:main}.
\end{rem}

We first state without proof a multidimensional analogue of Theorem \ref{thm:main}.
It can be proved, exactly in the same way as Theorem \ref{thm:main}, by combining the two purity theorems. 
Given a matrix $T$, we let $\rho(T)$ denote its spectral radius.  

\begin{thm}
\label{coro:ind_alg_1_multi_function}
Let $r\geq 2$ be an integer. For every integer $i$, $1\leq  i\leq r$, we let $\z_i=(z_{i,1},\ldots,z_{i,n_i})$ 
be a vector of indeterminates, 
$T_i$ be $n_i\times n_i$ matrix that belongs to $\mathcal M$, 
$f_i(\z_i)\in \Q\{\z_i\}$ be a regular singular $T_i$-Mahler function, and 
$\balpha_i=(\alpha_{i,1},\ldots,\alpha_{i,n_i})\in \Q^{n_i}$ be 
such that  the pair $(T_i,\balpha_i)$ is admissible and $\balpha_i$ is regular with respect to the 
underlying regular singular $T_i$-Mahler system.  
Then the following properties hold. 

\begin{itemize}

\item[{\rm (i)}] Let us assume that the point $(\balpha_1,\ldots, \balpha_r)$ is  $T$-independent, 
where $T={\rm diag}(T_1,\ldots,T_r)$. Then $f_1(\balpha_1),f_2(\balpha_2),\ldots,f_r(\balpha_r)$ are 
algebraically independent over $\Q$ if and only if they are all transcendental.  

\medskip

\item[{\rm (ii)}] Let us assume that the numbers $\rho(T_1),\ldots, \rho(T_r)$ are pairwise multiplicatively independent. 
Then $f_1(\balpha_1),f_2(\balpha_2),\ldots,f_r(\balpha_r)$ are algebraically independent over $\Q$ if and only if 
they are all transcendental.  
\end{itemize}
\end{thm}

\begin{rem}
We stress that if all the algebraic numbers $\alpha_{i,j}$, $1\leq i\leq r$, $1\leq j\leq n_i$, are multiplicatively 
independent, then the point $(\balpha_1,\ldots, \balpha_r)$ is $T$-independent for any choice of the matrices $T_i$.  
\end{rem}

Though Mahler method in several variables is mostly used to deal with analytic functions in one variable, 
some Mahler functions in several variables have their own interest.  
This is the case of the generating functions of multidimensional automatic sequences (see \cite[Chapter 14]{AS} for a definition). 
More precisely, if ${\bf a}=(a(n_1,\ldots,n_d))_{(n_1,\ldots,n_d)\in\mathbb N^d}$ is a $d$-dimensional  $q$-automatic sequence with values in $\Q$, 
then the generating function  
$$
f_{\bf a}(z_1,\ldots,z_d)=\sum_{(n_1,\ldots,n_d)\in\mathbb N^d}a(n_1,\ldots,n_d)z_1^{n_1}\cdots z_d^{n_d}
$$
is a $T$-Mahler function with 
$$
T=\left(\begin{array}{rcl}
q& & \\
&\ddots&\\
&& q
\end{array} \right)\,.
$$
Let us illustrate this fact with a simple example. 

\begin{ex}
The two-dimensional Sierpinski sequence ${\bf s}=s(n_1,n_2)$ is defined by $s(n_1,n_2)=1$ if the natural numbers 
$n_1$ and $n_2$ have no $1$ at the same position in their ternary expansion, and by $s(n_1,n_2)=0$ otherwise. 
The name of this sequence comes from the fact that 
replacing $1$'s with black squares and $0$'s with white squares, and suitably renormalizing, the graphic representation of $\bf s$ 
converges (for the Hausdorff topology) to the Sierpinski carpet. More generally, many classical fractals can be obtained by a similar process 
using multidimensional automatic sequence (see for instance \cite{AS,ABe}). 
Figure \ref{Fig:Sierpisnki_Carpet}  provides a finite automaton generating the sequence $\bf s$.  
It takes as input a pair of natural numbers $(n_1,n_2)$ written in base $3$ and then padding, if necessary, the expansion of $n_1$ or $n_2$ at the beginning with $0$'s 
to ensure that both expansions have the same length. 

\label{ex:Sierpinski}
\begin{figure}[h]
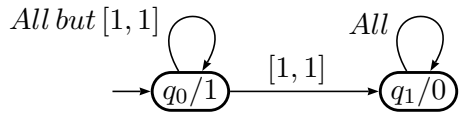

\centering
\VCDraw{%
        \begin{VCPicture}{(-2.5,-1)(2.5,2)}
 \StateVar[q_0/1]{(-2.5,0)}{a}  \StateVar[q_1/0]{(2.5,0)}{b}
\Initial[w]{a}
\LoopN{a}{All\, but\, [1,1]}
\LoopN{b}{All}
\EdgeL{a}{b}{[1,1]}
\end{VCPicture}
}
 \caption{The Sierpinski sequence Automaton}\label{Fig:Sierpisnki_Carpet}
\end{figure}

\medskip

The generating function $f_{\bf s}(z_1,z_2)=\sum_{n_1,n_2} s(n_1,n_2) z_1^{n_1} z_2^{n_2}$ is a 
$3{\rm I}_2$-Mahler function,  where we let denote by ${\rm I}_2$ denote the $2\times 2$ identity matrix. 
Indeed, it 
satisfies the regular singular equation
\begin{equation}
\label{eq:Sierpinski}
s(z_1,z_2)=(1+z_1+z_2+z_1^2+z_2^2+z_1^2z_2+z_1z_2^2+z_1^2z_2^2)s(z_1^3,z_2^3) \, .
\end{equation}
Set $a(z_1,z_2)=(1+z_1+z_2+z_1^2+z_2^2+z_1^2z_2+z_1z_2^2+z_1^2z_2^2)$. 
A point $\balpha=(\alpha_1,\alpha_2)\in(\Q^\star)2$ is $3{\rm I}_2$-independent if and only if $\alpha_1$ and $\alpha_2$ 
are multiplicatively independent. Furthermore, if $\alpha_1$ and $\alpha_2$ 
are multiplicatively dependent, the number $s(\alpha_1,\alpha_2)$ is the value of a one-dimensional $3$-Mahler function,
 obtained by specializing Equation \ref{eq:Sierpinski}. In the end, we can prove that outside the  Zariski closed
set $\{a(z_1,z_2)=0\}$, $s(\alpha_1,\alpha_2)$ is transcendental for all pair of algebraic numbers with $0 < |\alpha_1|,|\alpha_2|<1$. 
We can also use Theorem \ref{coro:ind_alg_1_multi_function} to prove for instance that $s(1/2,1/3)$ and $s(1/5,1/7)$ 
are algebraically independent over $\Q$.
\end{ex}

We recall that Semenov \cite{Se} obtained an interesting generalization of Cobham's theorem for $d$-dimensional automatic sets. 
In this direction, we extend Corollary \ref{coro:function} to Mahler functions in several variables.

\begin{thm}
\label{th:alg_ind_multivariable}
Let $n$ be a positive integer, and let $T_1,\ldots,T_r$ be $n\times n$ matrices in $\mathcal M$ such that $\rho(T_1),\ldots,\rho(T_r)$ 
are pairwise multiplicatively independent. For every $i$, $1\leq i\leq r$, let $f_i(\z)$ be a regular singular $T_i$-Mahler function that is  
not a rational function. Then  $f_1(\z),\ldots,f_r(\z)$ are algebraically independent over $\Q(\z)$.
\end{thm}

\begin{rem}
The case $r=1$ gives that a regular singular $T$-Mahler function is either rational or transcendental, providing 
that $T$ belongs to $\mathcal M$.
\end{rem}

We are now going to prove Theorem \ref{th:alg_ind_multivariable}.

\begin{lem}
\label{lem:existenceitereregulier}
Let $T_1,\ldots,T_r$ be $n\times n$ matrices that belong to $\mathcal M$. 
Let $\psi(\z)$ be a non-zero analytic function with coefficients in a number field $K$. 
There exist some algebraic numbers $\beta_1,\ldots,\beta_t$ such that, if $\balpha=(\alpha_1,\ldots,\alpha_n) \in (K^\star)^n$ is such that 
the numbers  
$\alpha_1,\ldots,\alpha_n,\beta_1,\ldots,\beta_t$ are multiplicatively independent and $1<\vert\alpha_1\vert,\ldots,\vert\alpha_n\vert<1$,  
then there exists a non-singular matrix $S$ with non-negative integer coefficients, such that 
$$
 \psi(T_i^{k}S\balpha)\neq 0\, ,
$$
for all $k\geq 1$ and all $i$ with $1 \leq i \leq r$.
\end{lem}

\begin{proof}
We let $\Vert\cdot\Vert$ denote the maximum norm (for both vectors with complex coordinates and square matrices with integer coefficients).   
The proof is based on Theorem 3 of \cite{CZ05}. 
This result implies that there exist a finite number of $n$-tuple of integers $\bmu_1,\ldots,\bmu_t$, and 
a finite number of algebraic numbers 
$\beta_1,\ldots,\beta_t$ such that
$$
\prod_{i=1}^t(\x_k^{\bmu_i} - \beta_i) = 0,\, \text{ for infinitely many } k \in \N
$$
for all sequences $(\x_k)_{k\in \N} \subset K^n$ satisfying the following conditions. 
\begin{enumerate}[label=(\Alph*)]
\item $\psi(\x_k)=0$ for infinitely many $k \in \N$.
\item \label{item:CZconvergence}$\lim_{k\to\infty}\x_k = 0$.  
\item \label{item:CZSunit} Every $\x_k$ is a $\mathcal S$-unit, for some finite set of places $\mathcal S$ over $K$. 
\item \label{item:CZhauteur} $h(\x_k)=\mathcal O( -\log ||\x_k ||)$, where we let $h(\cdot)$ denote the logarithmic Weil height.
 \end{enumerate} 
 
Let $\mathfrak M$ denote the monoid generated by the matrices $T_1,\ldots,T_r$ (with respect to usual matrix product). 
Since the matrices $T_1,\ldots,T_r$ belong to the class $\M$, for any $S \in \mathfrak M$ and $1 \leq i \leq r$, we have 
 $\Vert T_iS \Vert > \Vert S\Vert$. Furthermore, there are only finitely many matrices in $\mathfrak M$ with a given norm. 
We can thus define a total order $\succ$ on $\mathfrak M$ in the following way. Take $S,S' \in \mathfrak M$. If $||S|| > ||S'||$, 
 we say that $S \succ S'$, and if $||S_1|| = ||S_2||=\cdots =\Vert S_k\Vert$, 
 we choose an arbitrary order between these matrices.  
 We can thus consider a sequence $(S_k)_{k\geq 1}$ that enumerates the elements of $\mathfrak M$ according to $\succ$.   
 Let us consider a point $\balpha=(\alpha_1,\ldots,\alpha_n) \in (K^\star)^n$ such that 
the numbers  
$\alpha_1,\ldots,\alpha_n,\beta_1,\ldots,\beta_t$ are multiplicatively independent and $1<\vert\alpha_1\vert,\ldots,\vert\alpha_n\vert<1$.
Set $\x_k = S_k\balpha$, for every positive integer $k$. 
We also choose a finite set of places $\mathcal S$ over $K$, such that $\balpha$ is a $\mathcal S$-unit. 
Then, every $\x_k$ is also a $\mathcal S$-unit. 
We can estimate the logarithmic Weil height of the point $\x_k$. 
Let $\rho_1,\ldots,\rho_r$ denote the spectral radii of the matrices $T_1,\ldots,T_r$. 
Since $T_1,\ldots,T_r$ belong to $\mathcal M$, we get that 
$$
\log || \x_k || = \mathcal O(-\rho_1^{c_1}\cdots\rho_r^{c_r})\, ,
$$
where $c_i$ denote the number of occurrences of the matrix $T_i$ in a 
decomposition of $S_k$.  
On the other hand, we have 
$$
||S_k|| = \mathcal O(\rho_1^{c_1}\cdots\rho_r^{c_r}) \, .
$$
We refer the reader to \cite[Section 3]{AF3} for more details. 
Combining these two estimates, we get that 
$$
h(\x_k) \leq - \gamma \log || \x_k || \, ,
$$
for some positive real number $\gamma$. Furthermore, the way we define the order $\succ$ ensures that 
$\x_k \rightarrow 0$, as $k \rightarrow \infty$. Conditions \ref{item:CZconvergence} to \ref{item:CZhauteur} are thus satisfied. 
Let us assume now that $\psi(\x_k)=0$, for infinitely many $k$. Then, there exists an integer $j$, $1\leq j \leq t$, such that
\begin{equation}
\label{eq:xk}
(\x_k)^{\bmu_j} = \balpha^{\bmu_j S_k} = \beta_j\, ,
\end{equation}
for infinitely many integers $k \in \N$. This contradicts the fact that the numbers  
$\alpha_1,\ldots,\alpha_n,\beta_1,\ldots,\beta_t$ are multiplicatively independent.  
Thus, we deduce that there exists a positive integer $k_0$ such that $\psi(\x_k) \neq 0$ for all $k \geq k_0$. Set
$$
S=S_{k_0}\, .
$$
Since $T_i^kS \succ S$ for every $k \in \N$ and every $i$, $1\leq i \leq r$, we obtain that 
$$
\psi(T_i^kS\balpha) \neq 0\, ,
$$
for all $k\geq 1$ and all $i$ with $1 \leq i \leq r$. This ends the proof. 
\end{proof}

Recall that, for any matrix $T$ of class $\M$ we set $\mathcal U(T)$ the set of algebraic points of $(\Q^\star)^n$ 
for which condition (b) of definition 1.2 of the first part \cite{AF3} holds. 

\begin{lem}
\label{lem:densitetranscendence}
Let $T$ be a $n \times n$ matrix of class $\M$ and let $f(\z)$ be a regular-singular $T$-Mahler 
function that is not a rational function. Then there exists a Zariski closed set $\mathcal C$ of $\Q^n$ 
that contains all points $\balpha\in(\Q^\star)^n$ such that the pair $(T,\balpha)$ is admissible, $\balpha$ 
is regular with respect to the underlying regular singular Mahler system, and $f(\balpha)$ is algebraic.
\end{lem}

\begin{proof}
Let $(f_1(\z):=f(\z), f_2(\z),\ldots,f_m(\z))$ be a vector representing a solution to a regular singular $T$-Mahler system. 
Let $\balpha\in(\Q^\star)^n$ satisfying the assumptions of the lemma. 
The lifting theorem ensures the existence of polynomials $q_0(\z)q_1(\z),\ldots,q_m(\z) \in \Q[\z]$, not all zero, such that
\begin{eqnarray}
\label{eq:relationlinéaire}
q_0(\z)+q_1(\z)f_1(\z) + \cdots + q_m(\z)f_m(\z) &=& 0\, ,
\\
\label{eq:specialisation} q_0(\balpha) = f(\balpha),\ q_1(\balpha)=1,\ q_2(\balpha)=\cdots=q_m(\balpha)&=&0\, .
\end{eqnarray}
Let us consider the $\Q(\z)$-vector space of the linear relations over $\Q(\z)$ between the power series $\boldsymbol 1,f_1(\z),\ldots,f_m(\z)$. 
We choose a basis of this vector space, say $(r_{0,i}(\z),r_{1,i}(\z),\ldots,r_{m,i}(\z))\in\Q[z]^{m+1}$, $1 \leq i \leq l$. 
Let us consider the $l \times (m-1)$ matrix 
$$
R(\z)=\left(\begin{array}{ccc} r_{2,1}(\z) & \cdots & r_{m,1}(\z)
\\
\vdots & & \vdots
\\
r_{2,l}(\z) & \cdots & r_{m,l}(\z)
\end{array}
\right)
$$
Since $f_1(\z)=f(\z)$ is not a rational function, 
the matrix $R(\z)$ has rank $l$. 
Let $\Delta_1(\z),\ldots,\Delta_{m-1-l}(\z)$ denote the minors of $R(\z)$ of rank $l$. If $\balpha\in(\Q^\star)^n$ is such that $f(\balpha)$ is algebraic, 
Equalities \eqref{eq:relationlinéaire} and \eqref{eq:specialisation} implies the existence of a vector $\lambd(\z)=(\lambda_1(\z),\ldots,\lambda_l(\z))\in\Q[\z]^l$ 
such that 
$$
\lambd(\balpha)R(\balpha)=0\, .
$$
It follows that $\Delta_i(\balpha)=0$ for every $i$, $1\leq i \leq m-1-l$. 
Hence, $\balpha$ belongs to the Zariski closed set
$$
\mathcal C = \bigcap_{i=1}^{m-1-l} \{\balpha \in\Q^n\ \mid \ \Delta_i(\balpha)=0\}\, .
$$
This ends the proof. 
\end{proof}

We are now ready to prove Theorem \ref{th:alg_ind_multivariable}. 

\begin{proof}[Proof of Theorem \ref{th:alg_ind_multivariable}]
Our strategy is to find some suitable algebraic point $\balpha=(\alpha_1,\ldots,\alpha_n)$ at which we can 
apply the first purity theorem in order to obtain the algebraic independence of the numbers $f_1(\bbeta),\ldots,f_{r}(\bbeta)$ 
where $\bbeta=S\balpha$ for some suitable matrix $S$ with non-negative integer coefficients.    

 Let $\z=(z_1,\ldots,z_n)$ be a family of indeterminates, and $T_1,\ldots,T_r$ be $n\times n$ matrices that belong to $\M$ and 
 with spectral radii $\rho_1,\ldots,\rho_r$. We consider for each $i$, a $T_i$-Mahler power series $f_i(\z)$, from a system
\begin{equation}
\label{eq:Mahlermemetaille}\tag{\theequation .i}
\left(\begin{array}{c} f_{i,1}(\z) \\ \vdots \\ f_{i,m_i}(\z) \end{array}\right)
= A_i(\z)\left(\begin{array}{c} f_{i,1}(T_i\z) \\ \vdots \\ f_{i,m_i}(T_i\z) \end{array}\right)\, 
\end{equation}
with $f_i(\z)=f_{i,1}(\z)$, and $A_i(\z) \in {\rm Gl}_{m_i}(\Q(\z))$.
If $\rho<1$ is a sufficiently small positive real number, then the functions $f_{i,j}(\z)$, $1 \leq i \leq r$, $1 \leq j \leq m_i$ 
are all analytic on the open ball $B({\boldsymbol 0},\rho)$. In particular, $B(0,\rho) \subset \mathcal U(T_i)$ for every $i$. 
Define, for each $i$, a set $\mathcal C_i$ as in Lemma \ref{lem:densitetranscendence}, and choose $\Delta(\z)$ a polynomial 
such that if $\Delta(\balpha) \neq 0$ for some $\balpha \in B(0,\rho)$, then $\balpha$ does bot belong to any of the $\mathcal C_i$. 
We let $\delta(\z)$ denote a polynomial such that $\balpha$ 
is a regular point for the Mahler system \eqref{eq:Mahlermemetaille} if $\delta(T_i^k\balpha) \neq 0$ for every $k \geq 1$ and $1 \leq i \leq r$.  
We can take $\delta(\z)$ to be the product of the determinants, and of the denominators of the coefficients of the matrices $A_i(\z)$. 
Set 
$$
\psi(\z)=\Delta(\z)\delta(\z)\, ,
$$
and let $K$ be a number field containing the coefficients of $\psi(\z)$.  We choose $\balpha \in B(0,\rho) \cap K^n$ 
as in Lemma \ref{lem:existenceitereregulier}. By Lemma \ref{lem:existenceitereregulier}, 
there exists a $n \times n$ non-singular matrix $S$ with non-negative integer coefficients such that 
$$
\psi(T_i^kS\balpha) \neq 0\, ,
$$
for all $k\geq 1$ and all $i$ with $1 \leq i \leq r$. 
Since $S$ is non singular, and $\balpha$ has multiplicatively independent coordinates, the point 
$\bbeta:=S\balpha$ remains $T_i$-independent for every $i$. Hence, the pair $(T_i,\bbeta)$ is admissible. 
Since $\bbeta\in B(0,\rho)$, the functions $f_{i,j}(\z)$ are all well-defined at $\bbeta$. Since $\delta(T_i^k\bbeta) \neq 0$ for every $i$, 
the point $\bbeta$ is regular with respect to each system \eqref{eq:Mahlermemetaille}. 
We can thus apply the first purity theorem at $\bbeta$. 
Set $\mathcal E_i = \{f_i(\bbeta)\}$, $1 \leq i \leq r$, and $\mathcal E = \cup_{i} \mathcal E_i$. We have
$$
{\rm Alg}_{\Q}\left(\mathcal E\right)= \sum_{i=1}^r {\rm Alg}_{\Q}\left(\mathcal E_i \mid \mathcal E\right)\, .
$$
Since $\Delta(\bbeta)\neq 0$, the numbers $f_i(\bbeta)$ are transcendental, so that  
$ {\rm Alg}_\Q\left(\mathcal E_i \mid \mathcal E\right)=\{0\}$ for every $i$. 
Hence, ${\rm Alg}_\Q\left(\mathcal E\right)=\{0\}$ 
and the numbers $f_1(\bbeta),\ldots,f_{r}(\bbeta)$ are algebraically independent over $\Q$. 
In particular, the functions $f_1(\z),\ldots,f_r(\z)$ are algebraically independent over $\Q(\z)$. 
This ends the proof. 
\end{proof}

\section{Specializations of multivariate Mahler functions}\label{sec:spe}

As mentioned in the introduction, one interest of the multidimensional theory is to enlarge the class of 
one-dimensional analytic functions that fall under the scope of Mahler's method. 
In this section, we define the notion of a good $T$-Mahler specialization and prove an analogue of Corollary \ref{coro:function} 
and Theorem \ref{th:alg_ind_multivariable} for these functions. Then we discuss the connection with a nice extension of Cobham's theorem 
to morphic words obtained by Durand \cite{Du11}. 

\subsection{Good $T$-Mahler specializations and algebraic independence}

Given a quadratic irrational real number $\omega$, the Hecke--Mahler series $f_{\omega}(z)=\sum_{n\geq 0} \lfloor n\omega\rfloor z^n\in\Q\{z\}$ 
is a typical example of what we would like to think about as a \emph{good Mahler specialization}. Indeed, one has $f_{\omega}=F_{\omega}\circ \sigma$ where 
$F_{\omega}(z_1,z_2)=\sum_{n_1=0}^{\infty} \sum_{n_2=0}^{\lfloor n_1\omega\rfloor}z_1^{n_1}z_2^{n_2}$ is a two-dimensional Mahler function 
and  
$$
\sigma: \left\{\begin{array}{rcl} \Q & \rightarrow &\Q^{2}
\\
z & \mapsto & ( z,1)\, 
 \end{array} \right.
$$
is a polynomial map. 
We consider that this specialization is \emph{good} not only because it is the composition of a Mahler function by a polynomial map, 
but also because $F_{\omega}$ is a regular singular $T$-Mahler function for a suitable $2\times 2$ matrix $T\in \mathcal M$, 
and for all $\alpha\in\Q$, $0<\vert\alpha\vert<1$, the pair $(T,\sigma(\alpha))$ is admissible, and the point $\sigma(\alpha)$ is regular 
with respect to the Mahler system associated with $F_{\omega}$.  
These properties allow us to apply Mahler's method in several variable to the study of the values of $f_{\omega}$ at algebraic points.  
This leads us to the following definition. 

\begin{defi}\label{def:goodspecialization}
Let $\z=(z_1,\ldots,z_n)$ be a family of indeterminates and $T\in\mathcal M$ be a $n\times n$ matrix.  
 A \emph{good $T$-Mahler specialization} is a 
power series of the form $f\circ \sigma(z)\in\Q\{z\}$, where $f(\z)$ is a regular singular $T$-Mahler function 
and $\sigma$ is a map of the form 
$$
\sigma: \left\{\begin{array}{rcl} \Q & \rightarrow &\Q^{n}
\\
z & \mapsto & ( p_1(z),\ldots,p_m(z))\, 
 \end{array} \right.
$$
where $p_1(z),\ldots,p_m(z)$ are non-zero polynomials in $\Q[z]$, and which satisfies the following conditions. 

\begin{enumerate}
\item[{\rm (i)}] \label{item:admissibility} There exists a punctured neighborhood $\mathcal V$ of $0$ in $\Q$ 
such that $(T,\sigma(\xi))$ is admissible for all $\xi \in \V$. 

\item[{\rm (ii)}] \label{item:regularity} The point $\sigma(\xi)$ is regular for all $\xi \in \mathcal V$.
\end{enumerate}
\end{defi}

We prove the following analogue of Corollary \ref{coro:function} 
and Theorem \ref{th:alg_ind_multivariable} for good $T$-Mahler specializations. 

\begin{thm}
\label{th:algebraicindependancefunctions}
Let $T_1,\ldots,T_r$ be matrices in $\mathcal M$ such that $\rho(T_1),\ldots,\rho(T_r)$ are pairwise 
multiplicatively independent. For every $i$, $1\leq i\leq r$, let $g_i(z)$ be good $T_i$-Mahler specialization that is 
not a rational function. Then  $g_1(z),\ldots,g_r(z)$ are algebraically independent over $\Q(z)$.
\end{thm}

Theorem \ref{th:algebraicindependancefunctions} is a consequence of the following lemma. 

\begin{lem}
\label{lem:exspe}
Let $g(z)$ be a good $T$-Mahler specialization and let $\mathcal V$ be as in Definition \ref{def:goodspecialization}. Then the set
$$
\mathcal E_g := \left\{\xi\in\Q\cap \mathcal V \mid g(\xi)\in \Q\right\}
$$
is finite. 
\end{lem}

\begin{proof} 
 We first introduce 
 the following notation. We let $\z=(z_1,\ldots,z_n)$ 
be indeterminates, and we assume that $T$ is a square matrix of size $n$.   
By assumption, we can assume that 
$$
g(z) := f\circ \sigma(z) \in \Q\{z\}\, ,
$$
where $\sigma: \Q \rightarrow \Q^{n_i}$, and where $f(\z)=f_{1}(\z)$ is the first coordinate of a solution to the regular singular 
$T$-Mahler system   
\begin{equation}
\label{eq:Mahler}
\left(\begin{array}{c} f_{1}(\z) \\ \vdots \\ f_{m}(\z) \end{array}\right)
= A(\z)\left(\begin{array}{c} f_{1}(T\z) \\ \vdots \\ f_{m}(T\z) \end{array}\right)\, ,
\end{equation}
with $A(\z) \in {\rm Gl}_{m}(\Q(\z))$. 
Furthermore, there exists a punctured neighborhood of the 
origin $\mathcal V$ in $\Q$ such that, for every $\xi \in \mathcal V$, the pair $(T,\sigma(\xi))$ is admissible, 
and the point $\sigma(\xi)$ is regular with respect to the system \eqref{eq:Mahler}. 
Following the proof of Lemma \ref{lem:densitetranscendence}, we are going to build a proper 
Zariski closed set $\mathcal C$ of $\mathcal V$ containing every algebraic numbers $\xi \in \mathcal V$ such that $g(\xi)$ is 
algebraic, that is such that $\mathcal E_g\subset \mathcal C$. 
Let us consider the $\Q(\z)$-vector space of the linear relations over $\Q(\z)$ between the power series 
$\boldsymbol 1,f_{1}(\z),\ldots,f_{m}(\z)$.  Pick a basis of this vector space, say 
$(r_{0,j}(\z),r_{1,j}(\z),\ldots,r_{m,j}(\z))\in \Q[\z]^{m+1}$, $1 \leq j \leq l$. 
We consider the $l \times (m-1)$ matrix
$$
R(z)=\left(\begin{array}{ccc} r_{2,1}\circ \sigma(z) & \cdots & r_{m,1}\circ \sigma(z)
\\
\vdots & & \vdots
\\
r_{2,l}\circ \sigma(z) & \cdots & r_{m,l}\circ \sigma(z)
\end{array}
\right)
$$
Since by assumption $f\circ \sigma(z)=g(z)$ is irrational, the matrix $R(z)$ has rank $l$. 
Let $\xi \in \mathcal E_g$. Then the lifting theorem implies 
that $R(\xi)$ has rank strictly less than $l$. Setting $\mathcal C = \{\xi \in \mathcal V \ \mid \ {\rm rank}(R(\xi))<l \}$,  
we thus have that $\mathcal E_g\subset\mathcal C$.  
On the other hand, the definition of $\mathcal C$ shows that it is a one-dimensional Zariski closed set of $\mathcal V$, for 
the $r_{i,j}\circ \sigma(z)$ belong to $\Q[z]$. Furthermore, since $R(z)$ has rank $l$, $\mathcal C$ must be a proper subset of $\mathcal V$. 
As a proper one-dimensional Zariski closed set is always finite, we obtain that  
$\mathcal E_g$ is a finite set.   
This ends the proof. 
\end{proof}

\begin{proof}[Proof of Theorem \ref{th:algebraicindependancefunctions}] 
With each function $g_i$, we associate a $T_i$-Mahler function $f_i$, a map $\sigma_i$, and a set $\mathcal V_i$ as in 
Definition \ref{def:goodspecialization}. 
We also associate a set $\mathcal E_{g_i}$ as in Lemma \ref{lem:exspe}. Since each $\mathcal V_i$ is a punctured neighborhood 
of the origin in $\Q$, it follows that $\mathcal V_0:=\cap_{i=1}^r\mathcal V_i$ is infinite. 
Set $\mathcal E=\cup_{i=1}^r \mathcal E_{g_i}$. By Lemma \ref{lem:exspe}, $\mathcal E_0:=\mathcal E\cap \mathcal V_0$ is a finite set. 
Thus, the set  $\mathcal V_0\setminus \mathcal E_0$ is not empty. Let $\xi\in \mathcal V_0\setminus \mathcal E_0$. Then 
for every $1 \leq i \leq r$, the pair $(T_i,\sigma_i(\xi))$ is admissible, the point $\sigma_i(\xi)$ is regular with respect to the regular singular 
$T_i$-Mahler system associated with $f_i$, and the number $f_i(\sigma_i(\xi))$ is transcendental. 
Since $\rho(T_1),\ldots,\rho(T_r)$ are pairwise 
multiplicatively independent, we can apply the first purity theorem. We deduce that the numbers $g_1(\xi),\ldots,g_r(\xi)$ 
are algebraically independent over $\Q$. 
Hence, the power series $g_1(y),\ldots,g_r(y)$ are algebraically independent over $\Q(y)$. 
This ends the proof. 
\end{proof}

\begin{rem}
The proof of Theorem  \ref{th:algebraicindependancefunctions} shows that the same conclusion still 
holds true if we replace the assumption that each $\mathcal V_i$ is a punctured neighborhood of the origin in $\Q$ 
by the weaker assumption that the set $\cap_{i=1}^r\mathcal V_i$ is infinite. 
\end{rem}

\subsection{Morphic sequences, Cobham's theorem, and specializations}\label{sec:morph}

An alphabet $A$ is a finite set of symbols, also called letters. A finite word over $A$ is a 
finite sequence of letters in $A$  or, equivalently, an element of $A^*$, the free monoid  generated by $A$.  
We let denote by $\vert W\vert$ the length of a finite word $W$, that is, the number 
of symbols in $W$.  
If $a$ is a letter and $W$ a finite word,  
then $\vert W\vert_a$ stands for the number of occurrences of the letter $a$ in $W$. 
A map from $A$ to $A^*$ naturally extends to a map from $A^*$ into itself called an (endo)morphism.  
Given two alphabets $A$ and $B$, a map from $A$ to $B$ naturally extends to a map from $A^*$ into $B^*$ 
called a coding.  Let $q\geq 2$ be an integer. A morphism $\varphi$ over $A$ is said to be $q$-\emph{uniform} if 
$\vert \varphi(a)\vert =q$ for every letter $a$ in $A$, and simply  
\emph{uniform} if it is $q$-uniform for some $q$.  
A morphism $\varphi$ over $A$ is said to be prolongable 
on $a$ if $\varphi(a)=aW$ for some word $W$ and if the length of the word 
$\varphi^n(a)$ tends to infinity with $n$. Then the word 
$$
\varphi^{\omega}(a):= \lim_{n\to\infty} \varphi^n(a) = aW\varphi(W)\varphi^{2}(W)\cdots
$$
is the unique fixed point of $\varphi$ that begins with $a$. 
An infinite word obtained by iterating a prolongable morphism $\varphi$ is said to be pure morphic. 
The image of a pure morphic word under a coding is a {\it morphic word} or a {\it morphic sequence}. 
A useful object associated with a morphism $\varphi$ is the so-called {\it incidence matrix} of $\varphi$,  
denoted by $M_{\varphi}$.  We first need to choose an ordering of the elements of $A$, say 
$A=\{a_1,a_2,\ldots,a_d\}$, and then $M_{\varphi}$ 
is defined by 
$$
\forall i,j\in \{1,\ldots,d\},\;\; \left(M_{\varphi}\right)_{i,j}:= \vert \varphi(a_j)\vert_{a_i} \, .
$$ 
The choice of the ordering has no importance. 
If a morphic sequence is generated by a morphism $\varphi$ such that the spectral radius of $M_{\varphi}$ is equal to $\rho$, 
we say that this sequence is a \emph{$\rho$-morphic sequence}.  
It is known that a sequence is $q$-automatic if and only if it is $q$-morphic. 
A famous example of non-automatic morphic sequence is given by the so-called Fibonacci word
$$
{\boldsymbol \varphi} = 0100101001001010010100100101001001\cdots\,,
$$ 
which is defined as the unique fixed point of the morphism $\varphi$ defined over $\{0,1\}$ by 
$\varphi(0)=01$ and $\varphi(1)=0$. This word is $(1+\sqrt 5)/2$-morphic. 
Quite recently, Durand \cite{Du11} prove the following nice generalization of Cobham's theorem 
that was open for a while: if $\rho_1$ and $\rho_2$ are multiplicatively independent algebraic numbers, 
a sequence that is both $\rho_1$- and $\rho_2$-morphic is eventually periodic.  
With an infinite word ${\bf a}=a_0a_1\cdots$ over a finite alphabet, we can associate the generating function  
$$
f_{\bf a}:=\sum_{n=0}^{\infty} a_nz^n\,.
$$
Furthermore, ${\bf a}$ is eventually periodic if and only if $f_{\bf a}$ is a rational function. 
In the vein of Problems \ref{conj: strongf} and \ref{conj: strongf2}, we expect that Durand's theorem can be strengthened  as follows.  

\begin{conj}
Let $r\geq 2$ be an integer. Let $\rho_1,\ldots,\rho_r$ be pairwise multiplicatively independent algebraic numbers, 
and, for every $i$, $1\leq i \leq r$, let  ${\bf a}_i$ be a $\rho_i$-morphic word that is not eventually periodic.   
Then, the generating functions $f_{{\bf a}_1}(z),\ldots,f_{{\bf a}_r}(z)$ are algebraically independent over $\Q(z)$. 
\end{conj}

Theorem \ref{th:algebraicindependancefunctions} provides a first general result towards this conjecture. 
Indeed, Cobham \cite{Co68} described how the generating function of any $\rho$-morphic word 
can be obtained as a specialization of the form $f\circ \sigma$, where $f$ is a $T$-Mahler functions in several variables 
and $\sigma(z)=(z,\ldots,z)$. Furthermore, $\rho(T)=\rho$. 
However, we stress that these specializations are not always good in the sense of Definition \ref{def:goodspecialization}, 
for it may happen that either $T$ does not belong to $\mathcal M$ or that $f$ is not regular singular. We give below a few examples. 

\begin{ex}\label{ex:bm} 
Let us consider the Baum--Sweet sequence $\boldsymbol{\mathfrak bs}$. This is a $2$-automatic sequence defined 
as follows. Its $n$th term is equal to $1$ 
 if the binary expansion of $n$ contains no block of consecutive $0$'s of odd length, and it is equal to $0$ otherwise.   
Let $\varphi$ denote the morphism defined by $\varphi(0)=01$, $\varphi(1)=21$, $\varphi(2)=13$, $\varphi(3)=33$, and $\tau$ be the 
coding defined by $\tau(0)=1$, $\tau(1)=1$, $\tau(2)=1$, $\tau(3)=0$. The sequence $\boldsymbol{\mathfrak bs}$ is also the image by $\tau$ of the unique fixed point 
of $\varphi$ beginning by $0$. 
The generating function $f_{\boldsymbol{\mathfrak bs}}(z)$ is a good $T_0$-Mahler specialization, with $T_0=(2)$, in a somewhat trivial way.  
Indeed, $f_{\boldsymbol{\mathfrak bs}}$ is a regular singular $2$-Mahler function, as we have 
$$
\left(\begin{array}{c} f_{\boldsymbol{\mathfrak bs}}(z) \\ f_{\boldsymbol{\mathfrak bs}}(z^2) \end{array}\right) 
=
\left(\begin{array}{cc} 0 & 1 \\ 1 & -z \end{array}\right)
\left(\begin{array}{c} f_{\boldsymbol{\mathfrak bs}}(z^2) \\ f_{\boldsymbol{\mathfrak bs}}(z^4) \end{array}\right) \,.
$$
Furthermore, every $\alpha\in\Q$, with $0<\vert \alpha\vert<1$, is regular. 
\end{ex}

\begin{ex}\label{ex:fibonacci} 
Let us consider the Fibonacci word $\boldsymbol \varphi$. 
Setting $\z=(z_0,z_1)$,  Cobham's construction leads to the $T_1$-Mahler system 
\begin{equation}
\label{eq:MahlerFib}
\left(\begin{array}{c} f_0(\z) \\ f_1(\z) \end{array}\right)
=
A(\z)
\left(\begin{array}{c} f_0(T_1\z) \\ f_1(T_1\z)  \end{array}\right)\, ,
\end{equation}
where
$$
A(\z) = \left(\begin{array}{cc} 
1&1
\\
z_0& 0
\end{array}\right)\, \mbox{ and } T_1=\left(\begin{array}{ccc} 1 & 1 \\ 1 & 0\end{array}\right)\, .
$$
According to Cobham, we get that $f_{\boldsymbol \varphi}(z)=f_1(z,z)$.  
Let us show that the system \eqref{eq:MahlerFib} is regular singular (even if the matrix $A({\bf 0})$ is singular). 
Setting 
$$
\Phi(\z) = \left(\begin{array}{cc} f_0(\z) & \frac{1}{z_0z_1} - \frac{1}{z_0} \\ f_1(\z) & \frac{1}{z_1}-\frac{1}{z_0z_1} \end{array}\right)
$$
and
$$
B= \left(\begin{array}{cc} 1 & 0 \\ 0 & - 1 \end{array}\right) \, ,
$$
we obtain that
$$
\Phi(\z) B = A(\z)\Phi(T_1\z) \,.
$$
Furthermore, $\det \Phi(\z)\not=0$ for it has a nonzero coefficient in $(z_0z_1)^{-1}$ 
in its generalized Laurent series expansion. 
It follows that the system \eqref{eq:MahlerFib} is regular singular.  
Furthermore, we note that $T_1$ belongs to the class $\mathcal M$.  
Finally, for every algebraic number $\alpha$ in the punctured open unit disk of $\mathbb C$, the point 
$(\alpha,\alpha)$ is regular and $T_1$-independent. 
Hence, $f_{\boldsymbol \varphi}(z)$ is a good $T_1$-Mahler specialization. 
We also note that $\rho(T_1)=(1+\sqrt 5)/2$.
\end{ex}

\begin{ex}\label{ex:morph} 
Let  $\varphi$ denote the binary morphism defined by $\varphi(0)=0110$ and $\varphi(1)=101$, 
and let 
$$
{\boldsymbol{\mathfrak w}}= 011010110101101010110101101\cdots
$$ 
denote 
the unique fixed point of $\varphi$ beginning with $0$.   
Setting $\z=(z_0,z_1)$,  Cobham's construction leads to the  $T_2$-Mahler system 
\begin{equation}
\label{eq:MahlerW}
\left(\begin{array}{c} f_0(\z) \\ f_1(\z) \end{array}\right)
=
A(\z)
\left(\begin{array}{c} f_0(T_2\z) \\ f_1(T_2\z)  \end{array}\right)\, ,
\end{equation}
where
$$
A(\z) = \left(\begin{array}{cc} 
1+z_0z_1^2&z_1
\\
z_0+z_0z_1& 1+z_0z_1
\end{array}\right)\, \mbox{ and } T_2=\left(\begin{array}{ccc} 2 & 1 \\ 2 & 2\end{array}\right)\, .
$$
According to Cobham, we get that $f_{\boldsymbol{\mathfrak w}}(z)=f_1(z,z)$. 
Furthermore, we note that $T_2$ belongs to the class $\mathcal M$ and that the system \eqref{eq:MahlerW} 
is regular singular for $A({\bf 0})$ is the $2\times 2$ identity matrix.  
Finally, for every algebraic number $\alpha$ in the punctured open unit disk of $\mathbb C$, the point 
$(\alpha,\alpha)$ is regular and $T_2$-independent. Hence, 
 $f_{\boldsymbol{\mathfrak w}}(z)$ is a good $T_2$-Mahler specialization. 
 We also note that $\rho(T_2)=2+\sqrt 2$. 
\end{ex}

\begin{ex}\label{ex:tribonacci}
Let us give an example over a $3$-letters alphabet. Let 
$$
{\boldsymbol{\mathfrak{tr}}}=0102010010201010201001020102\cdots
$$ 
denote the Tribonacci word, that is the unique fixed point of the morphism $\varphi$ defined by 
$\varphi(0)=01$, $\varphi(1)=02$, and $\varphi(2)=0$. 
Setting $\z=(z_0,z_1,z_2)$,  Cobham's construction leads to the  $T_3$-Mahler system 
\begin{equation}
\label{eq:MahlersystemTribonacci}
\left(\begin{array}{c} f_0(\z) \\ f_1(\z) \\ f_2(\z) \end{array}\right)
=
A(\z)
\left(\begin{array}{c} f_0(T_3\z) \\ f_1(T_3\z) \\ f_2(T_3\z) \end{array}\right)\, ,
\end{equation}
where
$$
A(\z):=\left(\begin{array}{ccc} 1 & 1 & 1 \\ z_0 & 0 & 0 \\ 0 & z_0 & 0 \end{array}\right)\, .
$$
and 
$$
T_3=\left(\begin{array}{ccc} 1 & 1 & 0 \\ 1 & 0 & 1 \\ 1 & 0 & 0 \end{array}\right)\, .
$$
According to Cobham, we get that $f_{\boldsymbol{\mathfrak{tr}}}(z)=f_1(z,z,z)+2f_2(z,z,z)$.  
Let us show that the system \eqref{eq:MahlersystemTribonacci} is regular singular. 
We first consider the inhomogeneous Mahler equation 
\begin{equation}
\label{eq:Tribonacci-h}
\begin{split}
h(\z) = &  \bar{j} h(T\z) + j z_0z_1 h(T^2\z) + z_0^3z_1^2z_2 h(T^3\z) 
\\  & - z_0^{1/2}z_1^{1/2}z_2^{1/2} + \bar{j} z_0^{3/2}z_1^{3/2}z_2^{1/2}  + jz_0^{5/2} z_1^{3/2} z_2^{1/2} +  z_0^{11/2}z_1^{7/2}z_2^{3/2}\, ,
\end{split}
\end{equation}
where we let $j$ denote the unique cubic root of unity with positive imaginary part. 
Equation \eqref{eq:Tribonacci-h} has a ramified analytic solution $h(\z)\in\Q[[z_0^{1/2},z_1^{1/2},z_2^{1/2}]]$, which can be obtained as 
the limit of the recurrence defined by $h_0(\z)= - z_0^{1/2}z_1^{1/2}z_2^{1/2}$, and by 
\begin{equation*}
\begin{split}
h_{n+1}(\z) =& \bar{j} h_n(T\z) + j z_0z_1 h_n(T^2\z) + z_0^3z_1^2z_2 h_n(T^3\z) 
\\  & - z_0^{1/2}z_1^{1/2}z_2^{1/2} + \bar{j} z_0^{3/2}z_1^{3/2}z_2^{1/2}  + jz_0^{5/2} z_1^{3/2} z_2^{1/2} +  z_0^{11/2}z_1^{7/2}z_2^{3/2} \, .
\end{split}
\end{equation*}
Now, let us consider the ramified Laurent polynomial 
$$
l(\z)=z_0^{-1/2}z_1^{-1/2}z_2^{-1/2} + jz_0^{-1/2}z_1^{-1/2}z_2^{1/2}  
+ \bar{j}z_0^{-1/2}z_1^{1/2}z_2^{-1/2} + z_0^{1/2}z_1^{1/2}z_2^{-1/2}\, .
$$
It satisfies the inhomogeneous Mahler equation
\begin{equation}\label{eq:Tribonacci-l}
\begin{array}{rcl} \bar{j} l(T\z) + jz_0z_1 l(T^2\z) + z_0^3z_1^2z_2 l(T^3\z) & = & 
 l(\z) - z_0^{1/2}z_1^{1/2}z_2^{1/2} + jz_0^{5/2}z_1^{3/2}z_2^{1/2}
\\ & & + \bar{j}z_0^{3/2}z_1^{3/2}z_2^{1/2} + z_0^{11/2}z_1^{7/2}z_2^{3/2}\, .
\end{array}
\end{equation}
Setting $g(\z) = l(\z) + h(\z)$, we infer from \eqref{eq:Tribonacci-h} and \eqref{eq:Tribonacci-l} that $g(\z)$ 
satisfies the homogeneous $T$-Mahler equation
\begin{equation}
\label{eq:Tribonacci-g}
g(\z) = \bar{j} g(T\z) + j z_0z_1 g(T^2\z) + z_0^3z_1^2z_2 g(T^3\z)\, .
\end{equation}
Setting 
$$
\Phi(\z) = \left(
\begin{array}{ccc}
t_0(\z) & g(\z) & \overline{g(\z)}
\\
t_1(\z) & \bar{j}z_0g(T\z) & jz_0\overline{g(T\z)}
\\
t_2(\z) & j z_0^2z_1 g(T^2\z) & \bar{j}z_0^2z_1\overline{g(T^2\z)}
\end{array} \right) \, ,
$$
and
$$
B=\left(
\begin{array}{ccc}
1& 0& 0
\\
0& j & 0
\\
0 & 0 & \bar{j}
\end{array} \right) \, ,
$$
we then infer from \eqref{eq:Tribonacci-g} and \eqref{eq:MahlersystemTribonacci} that 
$$
\Phi(\z)B=A(\z)\Phi(T\z) \, .
$$
Furthermore, we can check that $\det \Phi(\z)\not=0$ for it has a nonzero coefficient in 
$(z_0z_1z_2)^{-1}$ in its generalized Laurent series expansion. 
Since $B$ is a constant matrix and since $\Phi$ has coefficients in 
$\Q\{z_0^{1/2},z_1^{1/2},z_2^{1/2}\}$, this shows that the Mahler system 
\eqref{eq:MahlersystemTribonacci} is regular singular. 
Furthermore, $T_3$ belongs to the class $\mathcal M$.  
Finally, for every algebraic number $\alpha$ in the punctured open unit disk of $\mathbb C$, the point 
$(\alpha,\alpha,\alpha)$ is regular and $T_3$-independent. 
Hence, $f_{\boldsymbol{\mathfrak{tr}}}(z)$ is a good $T_3$-Mahler specialization. 
We also note that 
$\rho(T_3)=\left(1+{\sqrt[{3}]{19+3{\sqrt {33}}}}+{\sqrt[{3}]{19-3{\sqrt {33}}}}\right)/3$  
is the Tribonacci number, that is the unique real root of the polynomial $x^3-x^2-x-1$.  
\end{ex}

\begin{ex}\label{ex:nonex} 
Let us also add two non-examples. 

\begin{itemize}

\medskip

\item[$\bullet$] Let $\boldsymbol{\mathfrak n}_1$ denote the unique fixed point beginning with $0$ 
of the morphism 
$\varphi$ defined by $\varphi(0)=012$, $\varphi(1)=12$, 
and $\varphi(2)=2$. According to Cobham's construction, we obtain that the generating function  
$f_{\boldsymbol{\mathfrak n}_1}(z)$ is a specialization 
of a $T_4$-Mahler function where 
$$
T_4=\left(\begin{array}{ccc} 1 & 0& 0\\ 1 & 1 & 0 \\ 1 & 1 & 1 \end{array}\right)\, .
$$
However, since $\rho(T_4)=1$, the matrix $T_4$ does not belong to $\mathcal M$ and we cannot conclude 
that $f_{\boldsymbol{\mathfrak n}_1}(z)$ is a good 
Mahler specialization. 

\medskip

\item[$\bullet$] Let $\boldsymbol{\mathfrak n}_2$ denote the unique fixed point of the morphism 
$\varphi$ defined by $\varphi(0)=02$, $\varphi(1)=02112$, 
and $\varphi(2)=0212$. According to Cobham's construction, 
we obtain that the generating function $f_{\boldsymbol{\mathfrak n}_2}(z)$ is a specialization 
of a $T_5$-Mahler function where 
$$
T_5=\left(\begin{array}{ccc} 1 & 1 & 1\\ 0 & 2 & 1 \\ 1 & 2 & 2 \end{array}\right)\, .
$$
The spectral radius of $T_5$ is larger than $1$ for $T_5$ is primitive, but $1$ is an eigenvalue of $T_5$. 
 Hence, $T_5$ does not belong to $\mathcal M$ and we cannot conclude that 
 $f_{\boldsymbol{\mathfrak n}_2}(z)$ is a good Mahler specialization. 
\end{itemize}
\end{ex}

As an illustration of Theorem \ref{th:algebraicindependancefunctions}, we deduce the following result. 

\begin{prop}
The power series $f_{\boldsymbol{\mathfrak bs}}(z)$, $f_{\boldsymbol \varphi}(z)$, 
$f_{\boldsymbol{\mathfrak w}}(z)$, and $f_{\boldsymbol{\mathfrak{tr}}}(z)$ 
are algebraically independent over $\Q(z)$.  
\end{prop}

\begin{proof}
As already mentioned in Examples \ref{ex:bm} ,\ref{ex:fibonacci}, \ref{ex:morph}, and \ref{ex:tribonacci}, 
these generating functions are good Mahler specializations. 
Since none of the words $\boldsymbol{\mathfrak bs}$, $\boldsymbol{ \varphi}$, ${\boldsymbol{\mathfrak w}}$, and 
${\boldsymbol{\mathfrak{tr}}}$ is eventually periodic, 
the generating funcitons $f_{\boldsymbol{\mathfrak bs}}(z)$, $f_{\boldsymbol \varphi}(z)$, 
$f_{\boldsymbol{\mathfrak w}}(z)$, and $f_{\boldsymbol{\mathfrak{tr}}}(z)$ 
are all irrational.  Furthermore, the corresponding spectral radii are $2$, $(1+\sqrt 5)/2$, $2+\sqrt 2$, and 
$\left(1+{\sqrt[{3}]{19+3{\sqrt {33}}}}+{\sqrt[{3}]{19-3{\sqrt {33}}}}\right)/3$. 
These numbers are pairwise multiplicatively independent. 
By Theorem \ref{th:algebraicindependancefunctions}, it follows that 
$f_{\boldsymbol{\mathfrak bs}}(z)$, $f_{\boldsymbol \varphi}(z)$, $f_{\boldsymbol{\mathfrak w}}(z)$, and $f_{\boldsymbol{\mathfrak{tr}}}(z)$ are algebraically independent over $\Q(z)$.   
\end{proof}

\section{Application to Hecke--Mahler series}\label{sec: HeckeMahler}

In this section, we prove Theorem \ref{th:HeckeMahler}, as well as two complementary results about values of 
Hecke--Mahler series. 
Let $\omega$ be a quadratic irrational real number. As already mentioned in the introduction, 
the values of the Hecke--Mahler series   
$$
f_{\omega}(z)=\sum_{n=0}^\infty \lfloor n\omega \rfloor z^n
$$
can be obtained as values of a $T$-Mahler function in two variables. The underlying transformation $T$ is related 
to the continued fraction expansion of the parameter $\omega$. Mahler \cite{Ma29} uses this fact to prove that, for all algebraic number 
$\alpha$, $0<\vert \alpha\vert<1$, the number $f(\omega,\alpha)$ is transcendental. When considering values of Hecke--Mahler 
series at different algebraic points, there are two main results due to Nishioka \cite{Ni94} and Masser \cite{Mas99}.

\begin{thm}[Ku. Nishioka, 1994]
\label{th:Nishioka}
Let $\omega_{1},\ldots,\omega_{r}$ be quadratic irrational real numbers such that the quadratic fields 
$\mathbb Q(\omega_1),\ldots,\mathbb Q(\omega_r)$ are all distinct. Let $\alpha$ be an algebraic number, 
with $0<|\alpha|<1$, and let $t_1,\ldots,t_r$ be positive integers. For every $i$, $1\leq i \leq r$, set $\alpha_i:=\alpha^{t_i}$. 
Then the numbers $f_{\omega_1}(\alpha_i),\ldots,f_{\omega_r}(\alpha_r)$ are algebraically independent over $\Q$.
\end{thm}

\begin{thm}[Masser, 1999]\label{th:Masser1}
Let $\omega$ be a quadratic irrational real number and let $\alpha_1,\ldots,\alpha_m$ be algebraic numbers 
with $0<|\alpha_1|,\ldots,|\alpha_m| < 1$. Then $f_{\omega}(\alpha_i),\ldots,f_{\omega}(\alpha_m)$ are
algebraically independent over $\Q$ if and only if $\alpha_1,\ldots,\alpha_m$ are distinct.
\end{thm}

Thus, Theorem \ref{th:HeckeMahler} is a generalization of both Theorem \ref{th:Nishioka} and Theorem \ref{th:Masser1}.
When considering the values of Hecke--Mahler series at a single point, Masser \cite{Mas99} obtained a complete result. 

\begin{thm}[Masser, 1999]\label{th:Masser2}
Let $\omega_1,\ldots,\omega_r$ be quadratic irrational real numbers, and let $\alpha$ be an algebraic number 
with $0<|\alpha| < 1$. Then $f_{\omega_1}(\alpha),\ldots,f_{\omega_r}(\alpha)$ are algebraically independent 
over $\Q$ if and only if $\pm \omega_1,\ldots,\pm \omega_n$ are distinct modulo the rational integers.
\end{thm}

Remark \ref{rem:contre-exemple_Hecke-Mahler} shows that a full generalization of Theorems \ref{th:Masser1} and \ref{th:Masser2} 
cannot hold true. However, in addition to Theorem \ref{th:HeckeMahler}, 
we can generalize Masser's theorems in two different ways. The first one deals with the nature of the possible relationships 
between values of Hecke-Mahler series.

\begin{thm}\label{th:HMnaturerelations}
Let $\omega_1,\ldots,\omega_r$ be quadratic irrational real numbers, and let $\alpha_1,\ldots,\alpha_r$ 
be distinct algebraic numbers with $0<|\alpha_1|,\ldots,|\alpha_r| < 1$. 
Then $f_{\omega_{1}}(\alpha_1),\ldots,f_{\omega_r}(\alpha_r)$ are algebraically dependent over $\Q$ 
if and only if $1,f_{\omega_{1}}(\alpha_1),\ldots,f_{\omega_r}(\alpha_r)$ are linearly dependent over $\Q$.
\end{thm}

The second one provides a complete picture when considering only two values of 
Hecke-Mahler series.  

\begin{thm}
\label{th:HM2fonctions}
Let $\omega_1$, $\omega_2$ be quadratic irrational real numbers, and $\alpha_1$, $\alpha_2$ be 
nonzero algebraic numbers with $0<|\alpha_1|,|\alpha_2| < 1$. The numbers $f_{\omega_1}(\alpha_1)$ and $f_{\omega_2}(\alpha_2)$ 
are algebraically dependent over $\Q$ if and only if $\alpha_1=\alpha_2$ and $\omega_1 = \pm \omega_2\, \mod\Z$.
\end{thm}

In order to prove Theorems \ref{th:HeckeMahler}, \ref{th:HMnaturerelations}, and \ref{th:HM2fonctions}, we first need 
the following lemma that combines Lemmas 3.3 and 7.3 of \cite{Mas99}.

\begin{lem}\label{lem:HMindependentpoints}
Let $\omega_1,\ldots,\omega_r$ be quadratic irrational real numbers such that $\mathbb Q(\omega_1)=\cdots = \mathbb Q(\omega_r)$, 
and let $\alpha_1,\ldots,\alpha_r$ be algebraic numbers with $0<\vert \alpha_1\vert,\ldots \vert \alpha_r\vert<1$ . 
Then there exist multiplicatively independent algebraic numbers $\beta_1,\ldots,\beta_s$, roots of unity $\zeta_1,\ldots,\zeta_r$, a postive integer $h$, 
a matrix $T$ of class $\M$ and of size $2s$, and analytic series $F_i(\z)=F_i(z_1,\ldots,z_{2s})$, such that the following hold. 

\medskip

\begin{itemize}

\item[{\rm (i)}] $F_i(x_1,1,x_2,1,\ldots,x_s,1)=f_{\omega_i}(\zeta_i^hM_i^h(\x))\mod \mathbb Q(\x)$, 
where $M_i$ is a monomial and $\x:=(x_1,\ldots,x_s)$.

\medskip

\item[{\rm (ii)}] $F_i(\bbeta)=f_{\omega_i}(\alpha_i) \mod \Q$, where $\bbeta:= (\beta_1,1,\beta_2,1,\ldots,\beta_s,1)$. 

\medskip

\item[{\rm (iii)}] $F_i(\z) = F_i(T\z) \mod \mathbb Q(\z)$.

\medskip

\item[{\rm (iv)}] The pair $(T,\bbeta)$ is admissible and the point $\bbeta$ is regular with respect to the regular singular Mahler Equation $\rm (iii)$.

\end{itemize}
\end{lem}

\begin{proof} 
For each quadratic irrational number $\omega$ we define, following section 3 of \cite{Mas99}, a $2\times 2$ matrix $T(\omega)$. 
This matrix belong to the class $\M$. Let $m$ be an integer, we let $T^{(m)}(\omega)$ denote 
the $m$-fold block of $T(\omega)$. Thus, we have $T^{(m)}(\omega)\in \M$. 
Furthermore, since $\mathbb Q(\omega_1)=\cdots = \mathbb Q(\omega_r)$, each pair of matrices $T(\omega_i), T(\omega_j)$ 
have multiplicatively dependent spectral radii. Hence, from \cite[Section 9]{Mas99}, it follows that each matrix $T(\omega_i)$ is 
conjugated to a positive power of $T(\omega_1)$, say $T_1$. 
Then, we infer from  \cite[Lemma 3.3]{Mas99} that there exist analytic series $G_1(z_1,z_2),\ldots,G_r(z_1,z_2)$ 
such that for every integer $i$, $1 \leq i \leq r$, one has 
\begin{equation}\label{eq:gi1}
G_i(z_1,z_2)=G_i(T_1(z_1,z_2)) \mod \Q(z_1,z_2)\, ,
\end{equation}
and
\begin{equation}\label{eq:gi2}
G_i(z,1)=f_{\omega_i}(z^h) \mod \Q(z)\, ,
\end{equation}
for some integer $h > 0$. We pick some numbers $\xi_1,\ldots,\xi_r$ such that $\xi_i^h=\alpha_i$.  
According to \cite[Lemma 3]{LvdP78} (see also \cite[section 3]{Mas99}), we can pick multiplicatively independent algebraic numbers 
$\beta_{1},\ldots,\beta_{s}$, roots of unity $\zeta_1,\ldots,\zeta_r$, and monomials $M_1,\ldots,M_r$ such that 
\begin{equation}\label{eq:mon}
\xi_i=\zeta_iM_i(\beta_1,\ldots,\beta_s)\, .
\end{equation}
Let us denote by $\x=(x_1,\ldots,x_s)$, $\y=(y_1,\ldots,y_s)$, and $\z=(x_1,y_1,\ldots,x_s,y_s)$ some vectors of indeterminates.  
We claim that the power series 
\begin{equation}\label{eq:FI}
F_i(\z):=G_i(\zeta_iM_i(\x),M_i(\y)) 
\end{equation}
satisfy the conditions of Lemma \ref{lem:HMindependentpoints}, with $T$ a positive power of $T_1^{(s)}$. 
Conditions (i) and (ii) follow from \eqref{eq:gi2}, \eqref{eq:mon}, and \eqref{eq:FI}.  
Let us show that Condition (iii) is satisfied. 
Let $d$ be a positive integer such that $\zeta_i^d=1$ for every $i$, $1\leq i\leq r$.  
Since $T_1$ has determinant $1$, there exists a positive integer $\ell$ such that $T_1^{\ell}$ is the identity matrix modulo $d$. 
Set $T=T_1^{(s)\ell}$. 
Then, we have
$$
F_i(T\z) = G_i(T(\zeta_iM_i(\x),M_i(\y))) \mod \Q(\z) \, 
$$
and \eqref{eq:gi1} implies that 
$$
F_i(T\z)=F_i(\z) \mod\Q(\z) \,.
$$  
Now, we check that Condition (iv) holds. 
By \cite[Lemma 3.2]{Mas99}, the point $\bbeta:=(\beta_1,1,\beta_2,1,\ldots,\beta_s,1)$ is $T$-independent.  
Furthermore, $T$ belongs to the class $\mathcal M$, so that the pair $(T,\bbeta)$ is admissible.  
Finally, since $F_i(\z)$ is 
well defined at $T^k\bbeta$ for every non-negative integer $k$, the point $\bbeta$ is regular with respect to the regular singular 
Mahler equation (iii). 
This ends the proof. 
\end{proof}

We are now ready to prove our three theorems on values of Hecke--Mahler series. 

\begin{proof} [Proof of Theorem \ref{th:HeckeMahler}]
Let us first fix an integer $i$, $1\leq i \leq r$. 
We infer from Lemma \ref{lem:HMindependentpoints} that there exist multiplicatively independent algebraic numbers $\beta_{i,1},\ldots,\beta_{i,s_i}$, 
roots of unity $\zeta_{i,1},\ldots,\zeta_{i,r_i}$, 
a matrix $T_i\in \M$ of size $2s_i$, and analytic series $F_{i,j}(\z)=F_{i,j}(z_1,\ldots,z_{2s_i})$, $1\leq j\leq m_i$, with the following properties. 

\medskip

\begin{itemize}

\item[{\rm (i)}] $F_{i,j}(\bbeta_i)=f_{\omega_i}(\alpha_j) \mod \Q$, where $\bbeta_i= (\beta_{i,1},1,\beta_{i,2},1,\ldots,\beta_{i,s_i},1)\in\Q^{2s_i}$. 

\medskip

\item[{\rm (ii)}] $F_{i,j}(\z) = F_{i,j}(T\z) \mod \mathbb Q(\z)$.

\medskip

\item[{\rm (iii)}] The pair $(T_i,\bbeta_i)$ is admissible and the point $\bbeta_i$ is regular with respect to the regular singular Mahler Equation $\rm (ii)$.

\end{itemize}

With the functions $F_{i,j}$, $1\leq j \leq m_i$, we can thus associate an almost diagonal Mahler system. 
The latter is regular singular for it is an upper triangular system with $1$ on the diagonal and 
the $F_{i,j}$ are analytic at the origin.  
Furthermore, according to \cite{Mas99} and \cite{Ni94}, the spectral radius of $T_i$ is a unit of  the ring of integers of $\Q(\omega_i)$.  
It follows that the spectral radii $\rho(T_1),\ldots,\rho(T_r)$ are pairwise 
multiplicatively independent.  Indeed if $\lambda_1$ and $\lambda_2$ are two quadratic irrational real numbers, and if 
$u_1$ (resp.\ $u_2$) is a unit of $\mathbb Q(\lambda_1)$ (resp.\  of $\mathbb Q(\lambda_2)$), 
then $u_1$ and $u_2$ are multiplicatively independent if and only if $\mathbb Q(\lambda_1)\neq\mathbb Q(\lambda_2)$. 
This is a consequence of the facts that all units of a real quadratic field are powers of a fundamental unit.  
We can thus apply the first purity theorem. We obtain that the numbers
$$
F_{i,j}(\bbeta_i)=f_{\omega_i}(\alpha_{i,j}),\, 1 \leq i \leq r,\, 1 \leq j \leq m_i\, ,
$$
are algebraically independent over $\Q$ if (and only if) for every integer $i$, $1 \leq i \leq r$, the numbers 
$$
F_{i,j}(\bbeta_i), 1 \leq j \leq m_i\, ,
$$
are algebraically independent over $\Q$. 
But by Theorem \ref{th:Masser1}, we already know that the numbers
$$
F_{i,j}(\bbeta_i)= f_{\omega_i}(\alpha_{i,j}),\, 1 \leq j \leq m_i\, ,
$$ 
are algebraically independent over $\Q$.   
This ends the proof. 
\end{proof}

\begin{proof}[Proof of theorem \ref{th:HMnaturerelations}]
Using Lemma \ref{lem:HMindependentpoints}, we obtain that each of the numbers $f_{\omega_{i}}(\alpha_i)$ can be obtained 
as the value of a regular singular Mahler functions $F_i(\z)$, at some admissible regular point $\bbeta_i$. 
Furthermore, these functions satisfy regular singular Mahler equations of the form
$$
F_i(\z)= F_i(T_i\z) + R_i(\z)
$$ 
where $R_i(\z)$ is a rational function. 
Now, let us assume that the numbers $f_{\omega_{1}}(\alpha_1),\ldots,f_{\omega_r}(\alpha_r)$ 
are algebraically dependent over $\Q$. 
By the first purity theorem, we obtain that there exists a subset $\mathcal I\subset \{1,\ldots,r\}$ such that 
the numbers $\{f_{\omega_{i}}(\alpha_i) \mid i\in\mathcal I\}$ are algebraically dependent over $\Q$ and such that 
$\mathbb Q(\omega_i)=\mathbb Q(\omega_j)$ for all $i,j\in \mathcal I$ . Indeed, if $\mathbb Q(\omega_i)\not=\mathbb Q(\omega_j)$, 
then, as already mentioned in the proof of Theorem \ref{th:HeckeMahler}, the spectral radii of $T_i$ and $T_j$ are multiplicatively independent. 
Without loss of generality, we assume that $\mathcal I= \{1,\ldots,\ell\}$. 
In that case, arguing as in the proof of Theorem \ref{th:HeckeMahler}, 
we can assume that $T_1=\cdots =T_\ell$ and $\bbeta_1=\cdots =\bbeta_\ell=:\bbeta$. Then we can apply the lifting theorem 
to these functions. We obtain that the functions 
$F_1(\z),\ldots,F_\ell(\z)$ are algebraically dependent over $\Q(\z)$. Then, we infer from \cite[Theorem 2]{LvdP77-3} that 
there is a $\Q$-linear combination of the $F_i(\z)$ that belongs to $\Q(\z)$. It follows that the numbers 
$$
1,F_1(\bbeta),\ldots,F_\ell(\bbeta)\, .
$$
 are linearly dependent over $\Q$. Since $F_i(\bbeta)=f_{\omega_i}(\alpha_i)$, for $1\leq i \leq \ell$, 
this ends the proof. 
\end{proof}

\begin{proof}[Proof of Theorem \ref{th:HM2fonctions}] 
We only have to prove the direct implication. 
Indeed, the converse implication is a direct consequence of either Theorem \ref{th:Masser1} or Theorem \ref{th:Masser2}. 
 The case where $Q(\omega_1)\neq Q(\omega_2)$ follows from Theorem \ref{th:HeckeMahler}. 
The case where $\alpha_1 = \alpha_2$ follows from Theorem \ref{th:Masser2}. 
Furthermore, when $\omega_1 =\pm \omega_2\, \mod \Z$, classical linear relationships between Hecke-Mahler series show that 
one can reduce the situation to the case where 
$\omega_1=\omega_2$. The latter follows from Theorem \ref{th:Masser1}. 
Finally, we can assume without any loss of generality 
that $Q(\omega_1)=Q(\omega_2)$, $\omega_1 \neq \pm \omega_2\, \mod \Z$, and  $\alpha_1\neq \alpha_2$. 

By Lemma \ref{lem:HMindependentpoints}, there exist a matrix $T\in\mathcal M$, a $T$-independent point 
$\bbeta=(\beta_1,1,\ldots,\beta_s,1)\in\Q^{2s}$, roots of unity $\zeta_1$, $\zeta_2$, a positive integer $h$, and 
regular singular $T$-Mahler functions $F_1(\z),\ F_2(\z)$ such that, for $i \in\{1,2\}$, we have 
\begin{equation}
\label{eq:HMcase2}
\begin{array}{rcl}
F_i(x_1,1,\ldots,x_r,1)&=&f_{\omega_i}(\zeta_i M_i(\x))\mod \Q(\x)\  \text{ and}
\\ F_i(\bbeta)&=&f_{\omega_i}(\alpha_i)\mod \Q\, ,
\end{array}
\end{equation}
where $M_1,\, M_2$ are monomials, and $\x=(x_1,\ldots,x_s)$\footnote{In fact, $\zeta_i$ is equal to the $\zeta_i^h$ of 
Lemma \ref{lem:HMindependentpoints}, and 
$M_i(\x)$ is equal to $M_i(\x)^h$ of Lemma \ref{lem:HMindependentpoints}.}.   
We claim that the functions $F_1$ and $F_2$ are algebraically independent over $\Q(\z)$. 
If not, we infer from \cite[Theorem 2]{LvdP77-3} that there exist two algebraic numbers $\lambda_1, \lambda_2$, not both zero, such that
$$
\lambda_1 F_1(\z) + \lambda_2F_2(\z) \in \Q(\z)\, .
$$
Then, according to \cite[Section 4]{Mas99}, Equality \eqref{eq:HMcase2} provides positive integers $t_1\leq t_2$, 
such that the sequence 
$$
a_k=\lambda_1 \left \lfloor \frac{k\omega_1}{t_1} \right\rfloor \zeta_1^{k/t_1} \mathds{1}_{t_1 \mid k} + 
\lambda_2 \left \lfloor \frac{k\omega_2}{t_2} \right\rfloor \zeta_2^{k/t_2} \mathds{1}_{t_2 \mid k} 
$$
satisfies a linear recurrence with constant coefficients. Then, \cite[Lemma 4.1]{Mas99} implies that 
the sequence 
$$
v_k=\lambda_1 \left \{ \frac{k\omega_1}{t_1} \right\} \zeta_1^{k/t_1}\mathds{1}_{t_1 \mid k} + 
\lambda_2 \left \{ \frac{k\omega_2}{t_2} \right\} \zeta_2^{k/t_2}\mathds{1}_{t_2 \mid k}
$$
is eventually periodic. Let us first assume that $t_1=t_2$. 
Let $u$ denote a positive integer such that $\zeta_1^u=\zeta_2^u=1$.  
The sequence 
$$
w_k:=v_{kut_1} = \lambda_1 \left \{ ku\omega_1 \right\} + \lambda_2 \left \{ ku\omega_2 \right\}
$$
is eventually periodic. By \cite[Lemma 8.1]{Mas99}, we obtain that $\lambda_1=\lambda_2=0$, a contradiction. 
Thus, we have $t_1 < t_2$. Choose a positive integer $r$ such that $t_2 \mid r$ and $v_{k+r}=v_k$ for $k$ large enough. 
Let $p$ be a prime number with $p > t_2$. Then, for every non-negative integer $\ell$, 
$t_2$ cannot divide $t_1p+rt_1\ell$.  
It follows that  
$$
w_\ell:=v_{t_1p+rt_1\ell} = \lambda_1 \left \{(p+r\ell)\omega_1 \right\}\zeta_1^{p+r\ell} \, .
$$ 
Since the sequence is $(w_\ell)_{\ell \geq 1}$  
is eventually periodic and $\omega_1$ is irrational, we obtain that $\lambda_1=0$. Hence, $\lambda_2=0$, a contradiction. 
Thus, the functions $F_1$ and $F_2$ are algebraically independent over $\Q(\z)$. 
We can gather the functions $F_1(\z)$ and $F_2(\z)$ into a single regular singular $T$-Mahler system. Furthermore, 
the pair $(T,\bbeta)$ is admissible and the point $\bbeta$ is regular with respect to this system. The lifting theorem implies  
 that the numbers $F_1(\bbeta)$ and $F_2(\bbeta)$ are algebraically independent. 
Then, we deduce from \eqref{eq:HMcase2} that $f_{\omega_1}(\alpha_1)$ and $f_{\omega_1}(\alpha_1)$ 
are algebraically independent. 
This ends the proof. 
\end{proof}

\begin{rem}
In the case where $\alpha_1$ and $\alpha_2$ are multiplicatively independent there is a simpler argument.   
Setting $F_i(z_1,z_2)=\sum_{n_1=0}^{\infty} \sum_{n_2=0}^{\lfloor n_1\omega_i\rfloor}z_1^{n_1}z_2^{n_2}$, $i\in\{1,2\}$,   
it follows from Mahler \cite{Ma29}  that $F_i(z_1,z_2)$ is a regular singular $T_i$-Mahler function, where $T_i$ belongs to the class $\mathcal M$, and such that  
\begin{equation}
\label{eq:HMmiseenequation}
F_i(z,1)=f_{\omega_i}(z)\,\;\; \forall i\in\{1,2\}\,.
\end{equation}
Furthermore, since $Q(\omega_1)=Q(\omega_2)$, the spectral radii $\rho(T_1)$ and $\rho(T_2)$ are multiplicatively dependent. 
Say that $\rho(T_1)^{m_1}=\rho(T_2)^{m_2}$. 
Since $\alpha_1$ and $\alpha_2$ are multiplicatively independent, the point $(\alpha_1,1,\alpha_2,1)$ is $T$-independent, 
where we let $T$ denote the diagonal block matrix whose blocks are $T_1^{m_1}$ and $T_2^{m_2}$. 
Thus, we can apply the second purity theorem, and we obtain that the numbers $F_1(\alpha_1,1)=f_{\omega_1}(\alpha_1)$ 
and $F_2(\alpha_2,1)=f_{\omega_2}(\alpha_2)$ 
are algebraically independent, as desired.
\end{rem}

\section{The matryoshka dolls principle}\label{sec:ex}

Let us consider the following general problem. Given algebraic numbers $\alpha_1,\ldots,\alpha_r$ 
and Mahler functions, or good specialization of Mahler functions,  
say $f_1(z),\ldots,f_r(z)$, we want to determine whether or not 
the numbers $f_1(\alpha_1),\ldots,f_r(\alpha_r)$ are algebraically independent over $\Q$. 
In this section, we illustrate how we can combine the two purity theorems and the lifting theorem 
to address this problem. 

Given an analytic function $f(z)$, we let $f^{(\ell)}(z)$ denote the $\ell$th derivative of the function $f(z)$. 
Let us consider the generating functions $f_{\boldsymbol{\mathfrak bs}}(z),f_{\boldsymbol \varphi}(z),  
f_{\boldsymbol{\mathfrak w}}(z)$, 
and $f_{\boldsymbol{\mathfrak{tr}}}(z)$ already introduced in Section \ref{sec:morph}.  
We also consider the generating function $f_{\boldsymbol{\mathfrak {tm}}}(z)$ of the Thue-Morse sequence 
and the generating function 
$f_{\boldsymbol{\mathfrak{pf}}}(z)$ of the regular paperfolding sequence. 
We recall that the sequence ${\boldsymbol{\mathfrak{tm}}}$ is defined in Example \ref{ex:ThueMorse}. 
Its generating function satisfies the regular singular inhomogeneous Mahler equation of order one 
\begin{equation}
\label{eq:eqtm}
f_{\boldsymbol{\mathfrak{tm}}}(z)=(1-z)f_{\boldsymbol{\mathfrak{tm}}}(z^2)+\frac{z}{1-z^2}\, \cdot
\end{equation}
The regular paperfolding sequence is yet another emblematic example of a $2$-automatic sequence.  
We recall that it can be defined as follows (see \cite[Example 5.1.6]{AS}).  
Let us take a rectangular piece of paper. Fold it in half lengthwise and then fold the result in half again. 
Keep on this procedure {\it ad infinitum}, taking care 
to make the folds always in the same direction.  
Unfolding the piece of paper and marking $1$ for the "hills", and $0$ for the "valleys",  
we obtain the regular paperfolding sequence  
$$
\boldsymbol{\mathfrak{pf}} = 11011001110010011101100011001001110\cdots\, . 
$$ 
Its generating function satisfies the regular singular inhomogeneous Mahler equation of order one 
\begin{equation}
\label{eq:paperfolding}
f_{\boldsymbol{\mathfrak{pf}}}(z)=f_{\boldsymbol{\mathfrak{pf}}}(z^2) + \frac{z}{1-z^4}\, .
\end{equation}

Now, we illustrate how to combine our three main results with the following example.  

\begin{prop}\label{prop:exemplefinal}
The numbers
$$
f_{\boldsymbol \varphi}\left(\frac{1}{2}\right), f_{\boldsymbol \varphi}\left(\frac{1}{3}\right),f_{\boldsymbol \varphi}\left(\frac{1}{5}\right),
f_{\boldsymbol{\mathfrak{tr}}}\left(\frac{1}{2}\right), f_{\boldsymbol{\mathfrak{tr}}}\left(\frac{1}{6}\right), f_{\boldsymbol{\mathfrak w}}\left(\frac{1}{3}\right),
f_{\boldsymbol{\mathfrak w}}\left(\frac{1}{7}\right) $$

$$
f_{\boldsymbol{\mathfrak{tm}}}\left(\frac{1}{2}\right), \left(f_{\boldsymbol{\mathfrak{tm}}}^{(\ell)}\left(\frac{1}{10}\right)\right)_{\ell\geq 1}, 
f_{\boldsymbol{\mathfrak{pf}}}\left(\frac{1}{2}\right), 
\left(f_{\boldsymbol{\mathfrak{pf}}}^{(\ell)}\left(\frac{1}{7}\right)\right)_{\ell\geq 1},\left(f_{\boldsymbol{\mathfrak{bs}}}^{(l)}\left(\frac{1}{3}\right)\right)_{l \geq 0}
$$
are algebraically independent over $\Q$.
\end{prop}

Before proving Proposition \ref{prop:exemplefinal}, we first need the following simple lemma. 

\begin{lem}
Let $q\geq 2$ and $\ell$ be two natural numbers, and let $f(z)$ be a regular singular $q$-Mahler function. 
Then the derivative $f^{(\ell)}(z)$ is a regular singular $q$-Mahler function.
\end{lem}

\begin{proof}
By assumption, the function $f(z)$ is the first coordinate of a column vector $\f(z)$ representing a solution 
of a regular singular Mahler system
\begin{equation}
\label{eq:qMahler}
\f(z)=A(z)\f(z^q)\, .
\end{equation}
Deriving this equality, we obtain the new system
\begin{equation}
\label{eq:qMahlerderivée}
\left(\begin{array}{c}\f(z) \\ \f'(z) \end{array}\right) =
\left(\begin{array}{cc}A(z)& \boldsymbol 0 \\ A'(z) & qz^{q-1} A(z) \end{array}\right)
\left(\begin{array}{c}\f(z^q) \\ \f'(z^q) \end{array}\right)
\, .
\end{equation}
Since system \eqref{eq:qMahler} is regular singular, there exist an invertible matrix $\Phi(z)$ 
with coefficients in $\widehat{\bK}=\cup_{d\geq 1}\Q\{z^{1/d}\}$, and a constant matrix $B$, such that
\begin{equation}\label{eq:gaugetransform}
B \Phi(z^q)=\Phi(z)A(z)\, .
\end{equation}
Deriving this equality, one obtains
\begin{equation}\label{eq:gaugetransformderivée}
q z^{q-1}B\Phi'(z^q)=\Phi'(z)A(z) + \Phi(z)A'(z)\, .
\end{equation}
Hence, combining \eqref{eq:gaugetransform} and \eqref{eq:gaugetransformderivée}, 
we get that 
$$
\left(\begin{array}{cc} B & \boldsymbol 0 \\ \boldsymbol 0 & qB \end{array} \right)
\left(\begin{array}{cc} \Phi(z^q) & \boldsymbol 0 \\ z^q \Phi'(z^q) & z^q\Phi(z^q) \end{array} \right)
= \hspace{6cm}$$
$$
\hspace{5cm}\left(\begin{array}{cc} \Phi(z) & \boldsymbol 0 \\ z \Phi'(z) & z\Phi(z) \end{array} \right)
\left(\begin{array}{cc}A(z)& \boldsymbol 0 \\ A'(z) & qz^{q-1} A(z) \end{array}\right)\, .
$$
Hence, the system \eqref{eq:qMahlerderivée} is regular singular and $f'(z)$ is a regular-singular $q$-Mahler function. 
Iterating this process, we obtain that all the derivatives of $f(z)$ are regular singular $q$-Mahler function, which ends the proof. 
\end{proof}

\begin{proof}[Proof of Proposition \ref{prop:exemplefinal}]  
We first notice that each number in Proposition \ref{prop:exemplefinal} is the value at a regular point $\alpha$ of a regular singular $T$-Mahler function 
such that the pair $(T,\alpha)$ is admissible and $T\in\mathcal M$. 
For the numbers $f_{\boldsymbol \varphi}\left(\frac{1}{2}\right)$, $f_{\boldsymbol \varphi}\left(\frac{1}{3}\right)$, and  $f_{\boldsymbol \varphi}\left(\frac{1}{5}\right)$
the spectral radius of the underlying transformation $T$ is equal to $(1+\sqrt{5})/2$. 
For the numbers $f_{\boldsymbol{\mathfrak{tr}}}\left(\frac{1}{2}\right)$ and $f_{\boldsymbol{\mathfrak{tr}}}\left(\frac{1}{6}\right)$  
the spectral radius of the underlying transformation $T$ is equal to the tribonacci number.  
For the numbers $f_{\boldsymbol{\mathfrak w}}\left(\frac{1}{3}\right)$ and 
$f_{\boldsymbol{\mathfrak w}}\left(\frac{1}{7}\right)$  
the spectral radius of the underlying transformation $T$ is equal to $2+\sqrt{2}$. 
For the numbers  
 $$
 f_{\boldsymbol{\mathfrak{tm}}}\left(\frac{1}{2}\right), \left(f_{\boldsymbol{\mathfrak{tm}}}^{(\ell)}\left(\frac{1}{10}\right)\right)_{\ell\geq 1}, 
f_{\boldsymbol{\mathfrak{pf}}}\left(\frac{1}{2}\right), 
\left(f_{\boldsymbol{\mathfrak{pf}}}^{(\ell)}\left(\frac{1}{7}\right)\right)_{\ell\geq 1},\left(f_{\boldsymbol{\mathfrak{bs}}}^{(l)}\left(\frac{1}{3}\right)\right)_{l \geq 0}
 $$
the spectral radius of the underlying transformation $T$ is equal to $2$. 
Furthermore, we stress that all these numbers are transcendental. This could be proved by using Mahler's method, but this also a direct consequence 
of the work of Bugeaud and the first author (see for instance \cite[Theorem 4]{ABu07}). 

\paragraph*{Applying the first purity theorem}
Since the numbers 
$$
\frac{1+\sqrt{5}}{2}, \frac{1+{\sqrt[{3}]{19+3{\sqrt {33}}}}+{\sqrt[{3}]{19-3{\sqrt {33}}}}}{3}, 2+\sqrt{2} \text{ and }2\, ,
$$
are pairwise multiplicatively independent, we can apply the first purity theorem. We deduce that the numbers in Proposition \ref{prop:exemplefinal} are algebraically independent over $\Q$, 
if, and only if, the following properties hold. 

\medskip

\begin{itemize}
\item[({\rm a})] The numbers $f_{\boldsymbol \varphi}\left(\frac{1}{2}\right)$, $f_{\boldsymbol \varphi}\left(\frac{1}{3}\right)$, and  $f_{\boldsymbol \varphi}\left(\frac{1}{5}\right)$ 
are algebraically independent over $\Q$. 

\medskip

\item[({\rm b})] The numbers $f_{\boldsymbol{\mathfrak{tr}}}\left(\frac{1}{2}\right)$ and $f_{\boldsymbol{\mathfrak{tr}}}\left(\frac{1}{6}\right)$ 
are algebraically independent over $\Q$. 

\item[({\rm c})]  The numbers $f_{\boldsymbol{\mathfrak w}}\left(\frac{1}{3}\right)$ and 
$f_{\boldsymbol{\mathfrak w}}\left(\frac{1}{7}\right)$  
are algebraically independent over $\Q$. 

\item[({\rm d})]  The numbers 
$$
 f_{\boldsymbol{\mathfrak{tm}}}\left(\frac{1}{2}\right), \left(f_{\boldsymbol{\mathfrak{tm}}}^{(\ell)}\left(\frac{1}{10}\right)\right)_{\ell\geq 1}, 
f_{\boldsymbol{\mathfrak{pf}}}\left(\frac{1}{2}\right), 
\left(f_{\boldsymbol{\mathfrak{pf}}}^{(\ell)}\left(\frac{1}{7}\right)\right)_{\ell\geq 1},\left(f_{\boldsymbol{\mathfrak{bs}}}^{(l)}\left(\frac{1}{3}\right)\right)_{l \geq 0}
 $$
are algebraically independent over $\Q$. 
\end{itemize}

\paragraph*{Applying the second purity theorem} 
Now, we use the fact that all the numbers we consider are transcendental. 
Since the numbers $1/2$, $1/3$, and $1/5$ are multiplicatively independent, the second purity implies directly that (a) is satisfied.
Since the numbers $1/2$ and $1/6$ are multiplicatively independent, the second purity implies directly that (b) is satisfied. 
Again, since the numbers $1/3$ and $1/7$ are multiplicatively independent, the second purity implies directly that (c) is satisfied. 
Finally, since the numbers $1/2$, $1/10$, $1/3$ and $1/7$  are multiplicatively independent, the second purity implies directly that (d) is satisfied if, and only if 
the following properties hold. 

\medskip

\begin{itemize}
\item[(${\rm d}_1$)] The numbers $f_{\boldsymbol{\mathfrak{tm}}}\left(\frac{1}{2}\right)$ and $f_{\boldsymbol{\mathfrak{pf}}}\left(\frac{1}{2}\right)$ 
are algebraically independent over $\Q$. 

\medskip

\item[(${\rm d}_2$)] The numbers  $\left(f_{\boldsymbol{\mathfrak{tm}}}^{(\ell)}\left(\frac{1}{10}\right)\right)_{\ell\geq 1}$ 
are algebraically independent over $\Q$. 

\item[(${\rm d}_3$)]  The numbers $\left(f_{\boldsymbol{\mathfrak{pf}}}^{(\ell)}\left(\frac{1}{7}\right)\right)_{\ell\geq 1}$ 
are algebraically independent over $\Q$. 

\item[(${\rm d}_4$)]  The numbers 
$\left(f_{\boldsymbol{\mathfrak{bs}}}^{(l)}\left(\frac{1}{3}\right)\right)_{l \geq 0}$ 
are algebraically independent over $\Q$. 
\end{itemize}

\paragraph*{Applying the lifting theorem} 
Let us first prove (${\rm d}_1$). By the lifting theorem, the result would follow if we can prove that the functions $f_{\boldsymbol{\mathfrak{tm}}}(z)$ and 
$f_{\boldsymbol{\mathfrak{pf}}}(z)$ are algebraically independent over $\Q(z)$. But this can be derived from \cite[Theorem 3.5]{Ni_Liv}, using 
that both satisfy an inhomogeneous $2$-Mahler equation of order one (see \eqref{eq:eqtm} and \eqref{eq:paperfolding}). 
Hence, (${\rm d}_1$) is satisfied. Properties (${\rm d}_2$) and (${\rm d}_3$) also follow from the lifting theorem once we know that 
$f_{\boldsymbol{\mathfrak{tm}}}(z)$ and $f_{\boldsymbol{\mathfrak{pf}}}(z)$ are hypertranscendental. 
Again thanks to Equations \eqref{eq:eqtm} and \eqref{eq:paperfolding}, this can be deduced from of a result of Ke. Nishioka \cite[Theorem 3]{Ni84}.  Hence, (${\rm d}_2$) and (${\rm d}_3$) are satisfied. 
In order to prove (${\rm d}_4$), we can use once again the lifting theorem, but now we need to know that the functions 
$\left(f_{\boldsymbol{\mathfrak{bs}}}^{(l)}\left(z\right)\right)_{l \geq 0}$ 
and the functions $\left(f_{\boldsymbol{\mathfrak{bs}}}^{(l)}\left(z^2\right)\right)_{l \geq 0}$ 
are all algebraically independent over $\Q$. Indeed,  
$f_{\boldsymbol{\mathfrak{bs}}}(z)$ satisfies a $2$-Mahler equation of order two and not just an 
inhomogeneous equation of order one as in the previous cases. 
The results we need is recently proved by Dreyfus, Hardouin, and Roques. This is precisely Theorem 4.3 in  \cite{DHR}.  
Hence, (${\rm d}_4$) is satisfied. 
This ends the proof. 
\end{proof}

Let us discuss briefly one difficulty that we deliberately avoid in Proposition \ref{prop:exemplefinal}, 
but that may appear in other similar examples. Let us consider the numbers 
\begin{equation}
\label{eq:dependentpoints}
f_{\boldsymbol{\mathfrak{tm}}}\left(\frac{1}{2}\right), f_{\boldsymbol{\mathfrak{pf}}}\left(\frac{1}{3}\right), 
\text{ and } f_{\boldsymbol{\mathfrak{bs}}}\left(\frac{1}{6}\right)\, .
\end{equation}
They all are the value of a $2$-Mahler function and the numbers $\frac{1}{2}, \frac{1}{3}$, and $\frac{1}{6}$ 
are multiplicatively dependent.  
In such a situation, none of our three main results can be directly used.   
However, we can overcome this deficiency as follows. Let us consider the bivariate functions 
$g_{\boldsymbol{\mathfrak{tm}}}(z_1,z_2)=f_{\boldsymbol{\mathfrak{tm}}}(z_1)$, 
$g_{\boldsymbol{\mathfrak{pf}}}(z_1,z_2)=f_{\boldsymbol{\mathfrak{pf}}}(z_2)$, and 
$g_{\boldsymbol{\mathfrak{bs}}}(z_1,z_2)=f_{\boldsymbol{\mathfrak{bs}}}(z_1z_2)$, so that 
$$
g_{\boldsymbol{\mathfrak{tm}}}\left(\frac{1}{2},\frac{1}{3}\right)=f_{\boldsymbol{\mathfrak{tm}}}\left(\frac{1}{2}\right), \,
g_{\boldsymbol{\mathfrak{pf}}}\left(\frac{1}{2},\frac{1}{3}\right)=f_{\boldsymbol{\mathfrak{pf}}}\left(\frac{1}{3}\right), \,
g_{\boldsymbol{\mathfrak{bs}}}\left(\frac{1}{2},\frac{1}{3}\right)=f_{\boldsymbol{\mathfrak{bs}}}\left(\frac{1}{6}\right)\,.
$$
These are $T$-Mahler functions where 
$$
T=\left(\begin{array}{cc} 2 & 0 \\ 0 & 2 \end{array}\right)\in\mathcal M\, ,
$$
and the point $(1/2,1/3)$ is regular and $T$-independent, so that we can apply the lifting theorem.  
It implies that the algebraic independence over $\Q$ between the three numbers given in \eqref{eq:dependentpoints} 
can be obtained by proving the algebraic independence over $\Q(z_1,z_2)$ of the functions  
$$
f_{\boldsymbol{\mathfrak{tm}}}(z_1), \, f_{\boldsymbol{\mathfrak{pf}}}(z_2),\, f_{\boldsymbol{\mathfrak{bs}}}(z_1z_2),\, 
\text{ and }f_{\boldsymbol{\mathfrak{bs}}}(z_1^2z_2^2)\, .
$$

\begin{rem} 
The trick we used in the previous example can be made more general, as observed by 
Loxton and van der Poorten in \cite[Lemma 3]{LvdP78}.  In fact, we already used it 
in Lemma \ref{lem:HMindependentpoints}. 
Let  $f_1(z),\ldots,f_r(z)$ be regular singular $q$-Mahler functions and let  $\alpha_1,\ldots,\alpha_r$ be distinct 
algebraic numbers. For every $i$, $1\leq i\leq r$, let us assume that $f_i(\alpha_i)$ is well-defined 
and that there 
is a regular singular $q$-Mahler system 
\begin{equation}\tag{\theequation .i}\label{eq:smfin} 
\left(\begin{array}{c} f_i(z)=f_{i,1}(z) \\ f_{i,2}(z) \\ \vdots \\ f_{i,m_i}(z) \end{array}\right)
=A_i(z)\left(\begin{array}{c} f_{i,1}(z^q) \\ f_{i,2}(z^q) \\ \vdots \\ f_{i,m_i}(z^q) \end{array}\right)\, .
\end{equation}
We also assume that $\alpha_i$ is regular with respect to \eqref{eq:smfin}.   
We infer from \cite[Lemma 3]{LvdP78} that there exist roots of unity $\zeta_1,\ldots \zeta_r$, 
multiplicatively independent points $\beta_1,\ldots,\beta_s$, and monomials $M_1,\ldots,M_r$ 
such that
$$
\alpha_i = \zeta_i M_i(\beta_1,\ldots,\beta_s), 1 \leq i \leq r\, .
$$
There also exist positive integers $h$ and $l$ such that $ \zeta_i^{q^{h+l}}= \zeta_i^{q^{h}}$ for every $i$. 
Iterating $h$ times each system \eqref{eq:smfin}, we obtain a linear relation between the value 
of the function $f_i(z)$ at $\alpha_i$ and the values of the functions $f_{i,1}(z),\ldots,f_{i,m_i}(z)$ at $\alpha_i^{q^h}$. 
Now, we consider the transformation $T=q^l {\rm I}_s \in \M$, where ${\rm I}_s$ denote the $s$-dimensional identity matrix. 
We set
$$
g_{i,j}(\z) = f_{i,j}( \zeta_i^{q^h} M_i(\z)),\ 1 \leq i \leq r,\ 1 \leq j \leq m_i\, ,
$$
where $\z=(z_1,\ldots,z_s)$. We have,
$$
\left(\begin{array}{c} g_{i,1}(\z) \\ \vdots \\ g_{i,m_i}(\z) \end{array}\right)
=A_i(\zeta_i^{q^h} M_i(\z))\left(\begin{array}{c}  g_{i,1}(T\z) \\ \vdots \\ g_{i,m_i}(T\z)\end{array}\right)\, 
$$
and
$$
g_{i,j}\left(\beta_1^{q^h},\ldots,\beta_s^{q^h}\right) = f_{i,j}\left(\alpha_i^{q^h}\right) \,.
$$
Thus, we can apply the lifting theorem to this new system.  We deduce that if the functions $g_{i,j}(\z)$, 
$1 \leq i \leq r$, $1 \leq j \leq m_i$, are algebraically independent over $\Q(\z)$, then the numbers 
$f_1(\alpha_1),\ldots,f_r(\alpha_r)$ are algebraically independent over $\Q$. 
\end{rem}


\end{document}